\documentclass[11pt]{amsart}

\title[Tribracket brackets and quantum Enhancements]
{Quantum enhancement polynomials associated with the canonical two-element tricbracket}

\author{Yuya Koda}
\address{
Department of Mathematics, Hiyoshi Campus, Keio University, 4-1-1, Hiyoshi, Kohoku, Yokohama, 223-8521, Japan~ \slash ~ 
International Institute for Sustainability with Knotted Chiral Meta Matter (WPI-SKCM$^2$), Hiroshima University, 
1-3-1 Kagamiyama, Higashi-Hiroshima, Hiroshima 739-8526, Japan}
\email{koda@keio.jp}

\author{Yuya Nishimura}
\address{
Department of Mathematics, Hiroshima University, 1-3-1 Kagamiyama, Higashi-Hiroshima, 739-8526, Japan}

\author{Yuka Sakamoto}
\address{
Department of Mathematics, Hiroshima University, 1-3-1 Kagamiyama, Higashi-Hiroshima, 739-8526, Japan}

\thanks{
Y. K. is supported by JSPS KAKENHI Grant Numbers JP20K03588, 
JP21H00978, JP23H05437 and JP24K06744.
}

\usepackage{amsthm}
\usepackage{mathrsfs}
\usepackage{latexsym}
\usepackage[dvipdfmx]{graphicx}
\usepackage[dvips]{psfrag}
\usepackage[dvips]{color}
\usepackage{xypic}
\usepackage[abs]{overpic}
\usepackage{caption}
\usepackage[all]{xy}
\usepackage[normalem]{ulem}
\usepackage{tikz}

\theoremstyle{plain}
\newtheorem*{theorem*}{Theorem}
\newtheorem*{lemma*} {Lemma}
\newtheorem*{corollary*} {Corollary}
\newtheorem*{conjecture*}{Conjecture}
\newtheorem{theorem}{Theorem}[section]
\newtheorem{lemma}[theorem]{Lemma}
\newtheorem{corollary}[theorem]{Corollary}
\newtheorem{proposition}[theorem]{Proposition}
\newtheorem{conjecture}[theorem]{Conjecture}
\newtheorem{definition}[theorem]{Definition}
\newtheorem{claim}{Claim}

\theoremstyle{remark}

\newtheorem*{claim*}{Claim}
\newtheorem*{notation}{Notation}
\newtheorem{example}{Example}
\newtheorem*{remark}{Remark}

\theoremstyle{definition}

\newtheoremstyle{citing}
  {}
  {}
  {\itshape}
  {}
  {\bfseries}
  {.}
  {.5em}
  {\thmnote{#3}}

\theoremstyle{citing}

\newcommand{\ZZ}{\mathbb{Z}}

\newcommand{\cl}{\operatorname{cl}}
\newcommand{\Zpoly}{\mathbb{Z}[x_1^{\pm1}, \ldots ,x_5^{\pm1}]}
\newcommand{\sign}{\mathrm{sign}}



\begin{document}

\maketitle

\begin{abstract}
Quantum enhancement polynomials are invariants for oriented links, 
defined in association with an algebraic structure called a tribracket. 
In this paper, we focus on the particular case of the canonical two-element tribracket. 
We prove that, in that case, the quantum enhancement polynomials can be recovered by five specific polynomials, 
which we refer to as the universal quantum enhancement polynomials. 
After presenting several notable properties of these polynomials, we show that 
they are strictly stronger than the Jones polynomial. 
Furthermore, we provide computational results for links with up to 10 crossings. 
\end{abstract}

\vspace{1em}

\begin{small}
\hspace{2em}  \textbf{2020 Mathematics Subject Classification}: 
57K10; 57K16, 57K12




\hspace{2em} 

\end{small}

\section*{Introduction}

The study of knots and links by coloring their diagrams with algebraic structures $X$, 
such as quandles and their variants, has a long-established history in knot theory and 
is gaining increasing importance in contemporary research.
In this context, colorings are defined as assignments of elements from $X$ 
to parts of the link diagrams (e.g., crossings, regions), with $X$ being structured 
so that the number of colorings remains invariant under the three Reidemeister moves. 
An invariant of links can then be derived by counting the number of colorings for a given link diagram. 
This type of invariant is known as a \emph{counting invariant}. The classical example of such an invariant is 
\emph{Fox tricoloring}, where 
$X$ is just the group $\ZZ / 3 \ZZ$. 

\emph{Enhancements} of coloring-counting invariants are obtained by assigning an additional "label" to each coloring, 
rather than simply counting the colorings as they are. 
Although the precise definition of these enhancements is omitted here, they provide more powerful invariants. 
For instance, \emph{quandle cocycle invariants} introduced by 
Carter, Jelsovsky, Kamada, Langford, and Saito \cite{CJKLS99} can be interpreted within this framework. 
For further details, see Elhamdadi and Nelson \cite{EN15}.

In \cite{Nie14}, Niebrzydowski introduced sets equipped with ternary operations, now known as \emph{Niebrzydowski tribrackets}, or simply \emph{tribrackets}. 
Let $X$ be a tribracket and $R$ a commutative ring. 
If two maps 
 $A$ and $B$ from $X^3$ to $R^\times$ satisfy certain conditions 
 (see Section \ref{sec: Tribracket colorings of link diagrams and their quantum enhancement polynomials} for details), 
 then the pair $(A, B)$ is called a \emph{tribracket bracket} with respect to $X$ and $R$. 
In \cite{ANR21}, 
extending the concept of biquandle brackets to tribrackets, Aggarwal, Nelson, and Rivera defined an invariant 
$\Phi_X^{(A, B)}$ of oriented links, known as the \emph{quantum enhancement polynomial} associated with 
$(X; A, B)$. 
This paper focuses on quantum enhancement polynomials in the specific case where 
$X = X_2 := \ZZ / 2 \ZZ$, with the ternary operation 
$[a , b , c] = a + b - c$.

Even when restricting to the tribracket $X_2$, 
there are infinitely many tribracket brackets associated with it, depending on the choice of the commutative ring 
$R$ and tribracket brackets with respect to $(X_2, R)$. 
However, we show that there are five \emph{universal tribracket brackets} 
$( A^{(i)} , B^{(i)} )$ ($i=1, \ldots, 5$) 
such that any tribracket bracket with respect to $X_2$ and any integral domain 
$R$ can be recovered from them (Theorem~\ref{thm:iff_cond}). 
Therefore, the properties of the quantum enhancement polynomials associated with any tribracket bracket for $X_2$ 
and any integral domain can be understood by examining the five polynomials $\Phi^{(A^{(i)}, B^{(i)})}_{X_2}$ ($i=1, \ldots, 5$).

Let $L$ be a link and $L_1, L_2 \subset L$ be (possibly empty) sublinks of $L$ such that 
$L = L_1 \cup L_2$ and $L_1 \cap L_2 = \emptyset$. 
The \emph{linking number} $\mathrm{lk} (L_1, L_2)$ of $L_1$ and $L_2$ is defined as the 
the sum of the signs of crossings between $L_1$ and $L_2$ in which the over arc belongs to $L_1$ in a diagram of $L$. 
For a link $L$, let $\text{LK}(L)$ denote the multiset $\{\text{lk}(L_1, L_2) \mid L = L_1 \cup L_2\}$. 
We show that the universal quantum enhancement polynomials $\Phi^{(A^{(i)}, B^{(i)})}_{X_2} $ satisfy the following properties: 

\begin{itemize}
\item
Let $L$ and $L'$ be links. 
If $\Phi^{(A^{(i)}, B^{(i)})}_{X_2} (L) = \Phi^{(A^{(i)}, B^{(i)})}_{X_2} (L')$ for some $i \in \{1, \ldots, 5 \}$, 
then $\text{LK}(L) = \text{LK}(L')$ (Corollary~\ref{cor:linking_num_multi}). 
\item
 Let  $L$ and $L'$ be links. 
 If $\Phi^{(A^{(5)}, B^{(5)})}_{X_2} (L) = \Phi^{(A^{(5)}, B^{(5)})}_{X_2} (L')$, then $J(L) = J(L')$, 
 where $J ( \cdot )$ denotes the Jones polynomial
(Theorem \ref{thm:jones_phi5}). 
Furthermore, the invariant 
$\Phi^{(A^{(5)}, B^{(5)})}_{X_2}$ is a strictly stronger than the Jones polynomial (Proposition \ref{prop_trivial_thistlethwaite}).
\item
Let $K$ and $K'$ be knots. 
Fix any $i \in \{1,2,3,4,5\}$. 
Then,  $\Phi^{(A^{(i)}, B^{(i)})}_{X_2}(K) = \Phi^{(A^{(i)}, B^{(i)})}_{X_2}(K')$ if and only if $J(K) = J(K')$
(Corollary \ref{cor: universal tribracket brackets invariants are euivalent to the Jones polynomial for knots}). 
\end{itemize}

We computed the universal quantum enhancement polynomials for the two or more component links with up to 10 crossings, 
as listed in LinkInfo (\cite{LM24}). 
Our findings include the following: 

\begin{itemize}
\item
For any $i \in \{1,2,3,4\}$, 
there exist links $ L$ and $L' $ such that
$\Phi^{(A^{(i)}, B^{(i)})}_{X_2}(L) \neq \Phi^{(A^{(i)}, B^{(i)})}_{X_2}(L')$, even though $(J(L), \text{LK}(L)) = (J(L'), \text{LK}(L'))$ 
(Proposition~\ref{prop:computer experiments for Phi and (LK, J)}). 
\end{itemize}

The paper is organized as follows: 
Section 1 reviews the basics of quantum enhancement polynomials. 
Section 2 provides examples of these polynomials' computations. 
Section 3 lists the fundamental properties of quantum enhancement polynomials, several of which will be essential for the following sections.  
Section 4 proves the existence of universal tribracket brackets for the canonical two-element tribracket. 
Section 5 explores the relationship between these universal tribracket brackets and other invariants, such as linking numbers and Jones polynomials. 
Finally, Section 6 presents computational results for the universal quantum enhancement polynomials and discusses potential conjectures based on these findings.

\section{Tribracket colorings of link diagrams and their quantum enhancement polynomials}
\label{sec: Tribracket colorings of link diagrams and their quantum enhancement polynomials}

In this section, we review the notion of tribracket colorings of link diagrams and their associated quantum enhancement polynomials. 
For more details on tribracket colorings of link diagrams, see \cite{NOO19, Nie14}, and for quantum enhancement polynomials, refer to \cite{ANR21, ANR24}.

\begin{definition}\label{def:tribracket}
  Let $X$ be a set. A ternary operation $[, ,] : X^3 \to X$ is called a \textit{$($horizontal$)$ tribracket} if it satisfies the 
  following conditions: 
    \begin{enumerate}
        \item For all $a,b,c \in X$ 
        \begin{enumerate}
            \item there exits a unique $x \in X$ such that $[a,b,x] = c$. 
            \item there exits a unique $x \in X$ such that $[a,x,b] = c$.
            \item there exits a unique $x \in X$ such that $[x,a,b] = c$.
        \end{enumerate}
        \item For all $a,b,c,d \in X$, we have:
        \[
            [c, [a,b,c], [a,c,d]] = [b,[a,b,c], [a,b,d]] = [d, [a,b,d], [a,c,d]]    
        \]
    \end{enumerate}
\end{definition}

A tribracket $(X, [, ,])$  is sometimes denoted simply by $X$ if there is no risk of confusion. 

\begin{remark}
In \cite{NOO19, Nie14}, the ternary operation $[, , ]$ satisfying the conditions above itself is called a (horizontal) tribracket, 
and the pair $(X,\ [, ,])$ is referred to as a knot-theoretic horizontal ternary quasigroup, or a Niebrzydowski tribracket. 
For simplicity, we call it a tribracket here. 
\end{remark}

\begin{example}\label{ex:jones}
   Let  $X = \ZZ / n \ZZ$, and define a ternary operation $[, ,] : X^3 \to X$ by $[a,b,c] = a+b-c$. 
Then $(X, [, ,])$ is a tribracket. 
We call this the \textit{canonical tribracket} with $n$ elements and denote it by $X_n$.
\end{example}

In this paper, all links are assumed to be oriented unless otherwise specified. 
Let $D_L \subset \mathbb{R}^2$ be a diagram of a link $L$. 
Each component of $\mathbb{R}^2 - D_L$ is referred to as a \textit{region} of $D_L$. 
We denote by $\mathcal{C}(D_L)$ and $\mathcal{R}(D_L)$ the set of crossings and the set of regions of $D_L$, respectively.

\begin{definition}
Let $D_L$ be a diagram of a link $L$, and let $(X,\ [, ,])$ be a finite tribracket. 
A map $C : \mathcal{R}(D_L) \to X$ is called an \textit{$X$-coloring} of $D_L$ if it satisfies the rule shown in Figure~\ref{fig:1} around each crossing of $D_L$.
\begin{figure}[htbp]
\centering\includegraphics[width=8cm]{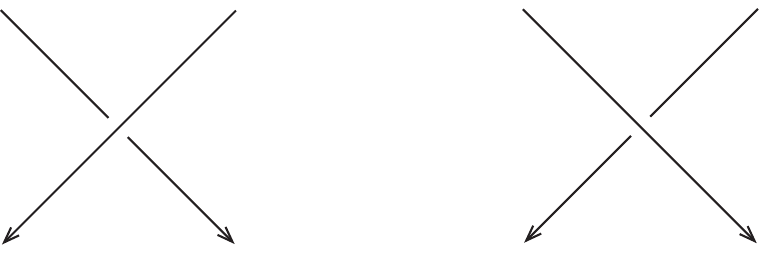}
\begin{picture}(400,0)(0,0)
\put(83,0){{\rm Positive crossing}}
\put(235,0){{\rm Negative crossing}}
\put(90,53){$x_c$}
\put(118,80){$y_c$}
\put(118,25){$z_c$}
\put(148,53){$[x_c, y_c, z_c]$}
\put(245,53){$x_c$}
\put(275,80){$z_c$}
\put(275,25){$y_c$}
\put(305,53){$[x_c, y_c, z_c]$}
\end{picture}
\caption{The rule for a region coloring around each vertex.}
\label{fig:1}
\end{figure}
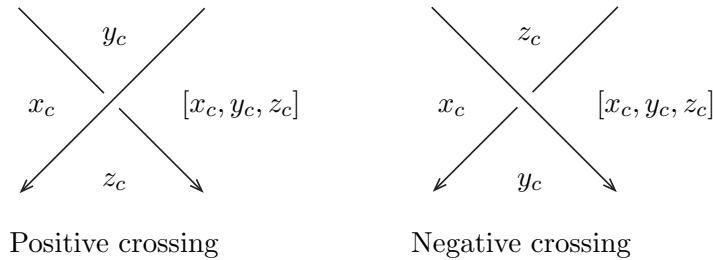
The set of $X$-colorings of $D_L$ is denoted by $\text{Col}_X(D_L)$. 
The cardinality of $\text{Col}_X(D_L)$ is called the \textit{$X$-coloring number} of the diagram $D_L$. 
We define a map $\operatorname{cl}_C : \mathcal{C}(D_L) \to X^3$ by $\operatorname{cl}_C(c) = (x_c, y_c, z_c)$, 
where $x_c, y_c, z_c$ are the colors of the regions around $c \in \mathcal{C}(D_L)$ as shown in Figure~\ref{fig:1}. 
\end{definition}

Let $D_L$ be a diagram of a link $L$, and let $C$ be an $X$-coloring of $D_L$. Suppose $D_L'$ 
is the diagram obtained by performing a directed Reidemeister move on $D_L$. 
Then, there exists a unique coloring of $D_L'$ that agrees with $C$ outside the local part of 
the diagram where the move is performed \cite{Nie14}. 
Therefore, the $X$-coloring number of a link diagram defines an invariant of oriented links, 
which we call the \textit{$X$-coloring number} of $L$. The process of obtaining $(D_L', C')$ from $(D_L, C)$ 
as described above is referred to as an \textit{$X$-colored Reidemeister move}.

Let $C$ be an $X_2$-coloring of a link diagram $D_L$. We say that $C$ is \textit{trivial} if $C$ is a constant map. 
On the other hand, $C$ is called a \textit{checkerboard $X_2$-coloring} if, for any region $r \in \mathcal{R}(D_L)$, 
the colors of adjacent regions to $r$ are all different from $C(r)$.

\begin{definition}[Aggarwal--Nelson--Rivera \cite{ANR21,ANR24}]\label{def:tribracketbracket}
Let $X$ be a tribracket and $R$ a commutative ring. A pair of maps $(A, B)$, where 
$A, B : X^3 \to R^\times$ and $A(a, b, c) = A_{a,b,c}$, $B(a, b, c) = B_{a,b,c}$ (with $R^\times$ denoting the group of units of $R$), 
is called a \textit{tribracket bracket with respect to $X$ and $R$} if it satisfies the following conditions:
    \begin{enumerate}
        \item \label{1} There exists an element $\delta \in R$ such that for all $a, b, c \in X$, we have
        \[\delta = - A_{a,c,b} B_{a,b,c}^{-1} - A_{a,b,c}^{-1} B_{a,b,c} . \]
        \item \label{2}
        There exists an element  $w \in R$, called the \textit{distinguish element} of $(A, B)$, such that for all $a,b \in X$, we have  
        \[ w = -A_{a,b,b}^2 B_{a,b,b}^{-1} . \]
    \item For all $a,b,c,d \in X$, we have 
    \begin{enumerate}
        \item $A_{a,b,c} A_{c, [a,b,c], [a,c,d]} A_{a,c,d} = A_{b, [a,b,c], [a,b,d]} A_{a,b,d} A_{d, [a,b,d], [a,c,d]}$ \label{3a}
        \item $A_{a,b,c} B_{c, [a,b,c], [a,c,d]} B_{a,c,d} = B_{b, [a,b,c], [a,b,d]} B_{a,b,d} A_{d, [a,b,d], [a,c,d]}$ \label{3b}
        \item $B_{a,b,c} B_{c, [a,b,c], [a,c,d]} A_{a,c,d} = A_{b, [a,b,c], [a,b,d]} B_{a,b,d} B_{d, [a,b,d], [a,c,d]}$ \label{3c}
        \item $\begin{aligned}
            A_{a,b,c} B_{c, [a,b,c], [a,c,d]} A_{a,c,d} = &A_{b, [a,b,c], [a,b,d]} A_{a,b,d} B_{d, [a,b,d], [a,c,d]}\\
                                                        &+ B_{b, [a,b,c], [a,b,d]} A_{a,b,d} A_{d, [a,b,d], [a,c,d]}\\
                                                        &+ \delta B_{b, [a,b,c], [a,b,d]} A_{a,b,d} B_{d, [a,b,d], [a,c,d]}\\
                                                        &+ B_{b, [a,b,c], [a,b,d]} B_{a,b,d} B_{d, [a,b,d], [a,c,d]} \label{3d}
        \end{aligned}$
        \item $\begin{aligned}
            &A_{a,b,c} A_{c, [a,b,c], [a,c,d]} B_{a,c,d} \\
            &+B_{a,b,c} A_{c, [a,b,c], [a,c,d]} A_{a,c,d} \\
            &+\delta B_{a,b,c} A_{c, [a,b,c], [a,c,d]} B_{a,c,d} \\
            &+B_{a,b,c} B_{c, [a,b,c], [a,c,d]} B_{a,c,d} = A_{b, [a,b,c], [a,b,d]} B_{a,b,d} A_{d, [a,b,d], [a,c,d]}. \label{3e}
        \end{aligned}$
    \end{enumerate}
    \end{enumerate}
\end{definition}

\begin{definition}[Aggarwal--Nelson--Rivera \cite{ANR21}]\label{def:beta}

Let $X$ be a tribracket, $R$ a commutative ring, and $(A, B)$ a tribracket bracket with respect to $X$ and $R$. 
Let $\delta, w \in R$ be as in Definition~\ref{def:tribracketbracket}. 
Let $D_L$ be a diagram of a link $L$. 
A map $s : \mathcal{C}(D_L) \to \{A, B\}$ is called a \textit{state} of $D_L$, and the set of states of $D_L$. 
is denoted by $\mathcal{S}(D_L)$. 
For each state $s$ of $D_L$, $k_s$ denotes the number of circles in the diagram obtained by smoothing each crossing as in 
Table~\ref{tab:smoothing}. 
    \begin{table}[htbp]
        \centering
        \begin{tabular}{c|c|c}
            & $\sign (c) = +1$ & $\sign (c) = -1$\\
            \hline
            \raisebox{7mm}{$s(c) = A$} & \includegraphics{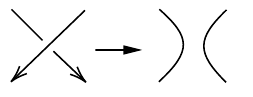} & \includegraphics{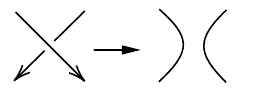}\\
            \hline
            \raisebox{7mm}{$s(c) = B$} & \includegraphics{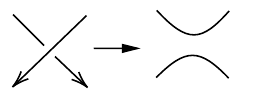} & \includegraphics{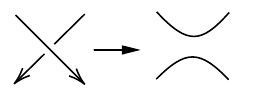}
        \end{tabular}
        \caption{Smoothing at a crossing $c$}
        \label{tab:smoothing}
    \end{table}
Let $C$ be an $X$-coloring of $D_L$. 
We then define an element $\beta_X^{(A, B)}(D_L, C) \in R$ by 
\[
\beta_X^{(A, B)}(D_L, C) = w^{n-p} \sum_{s \in \mathcal{S}(D_L)} \delta^{k_s} \prod_{c \in \mathcal{C}(D_L)} s(c)(\cl_C (c))^{\sign(c)}, 
 \]
where $p$ and $n$ denote the number of positive and negative crossings of $D_L$, respectively, and $\sign(c)$ denotes the sign of a crossing $c$. 
When $X$ and $(A, B)$ are clear from the context, we simply write $\beta(D_L, C)$ to refer to $\beta_X^{(A, B)}(D_L, C)$. 
\end{definition}

\begin{theorem}[Aggarwal--Nelson--Rivera \cite{ANR21}]\label{theorem:invariance of beta by Aggarwal--Nelson--Rivera}
Let $X$ be a tribracket and $(A, B)$ a tribracket bracket with respect to $X$ and a commutative ring. 
Then, for a diagram $D_L$ of $L$ and an $X$-coloring $C$ of $D_L$, $\beta_X^{(A, B)}(D_L, C)$ 
is invariant under $X$-colored Reidemeister moves. 
\end{theorem}

\begin{definition}[Aggarwal--Nelson--Rivera \cite{ANR21}]\label{def:phi}
Let $X$ be a tribracket and $(A, B)$ a tribracket bracket with respect to $X$ and a commutative ring. 
For each link $L$, the formal polynomial $\Phi_X^{(A, B)}(L)$ is defined by
    \[
        \Phi_X^{(A, B)}(L) = \sum_{C \in \text{Col}_X(D_L)} u^{\beta^{(A, B)}_X(D_L,C)}, 
    \]
where $D_L$ is a diagram of $L$. 
We call this the \textit{quantum enhancement polynomial} associated with $(X; A, B)$. 
Note that this is an invariant of $L$ by Theorem~\ref{theorem:invariance of beta by Aggarwal--Nelson--Rivera}. 
\end{definition}

\begin{example}
Let $(X, [, ,])$ be the \textit{trivial tribracket} defined by $X = \{0\}$ and $[0,0,0] = 0$. 
Let $R := \ZZ[x, x^{-1}]$ be the Laurent polynomial ring in one variable. 
If we define $A, B : X \to R^\times$ by $A(0,0,0) = x$ and $B(0,0,0) = x^{-1}$, then $(A, B)$ is a tribracket bracket with respect to $X$ and $R$.

Let $D_L$ be a diagram of $L$, and $C$ be the trivial $X$-coloring of $D_L$. By \cite{ANR21}, we have $\beta_X^{(A, B)}(D_L, C) = K(L)$, where $K(L)$ is the Kauffman bracket polynomial normalized such that $K(O) = -x^2 - x^{-2}$, where $O$ denotes the trivial knot. Therefore, if we substitute $x = t^{-\frac{1}{4}}$ into $\beta_X^{(A, B)}(D_L, C)$, we obtain the Jones polynomial $J(L)$ normalized as $J(O) = -t^{\frac{1}{2}} - t^{-\frac{1}{2}}$. Since the trivial $X$-coloring is the unique $X$-coloring of $D_L$, the invariant $\Phi_X^{(A, B)}$ is actually equivalent to the Jones polynomial. 
\end{example}

\section{Examples}

In this section, we provide formulas for the quantum enhancement polynomials associated with the canonical two-element tribracket 
$X_2$ for $(2, q)$ torus links and generalized twist links.

For an integer $k$, a box labeled $k$ denotes $k$ half-twists, as shown in Figure~\ref{fig:twist}.
\begin{figure}[htbp]
\centering\includegraphics[width=0.9\textwidth]{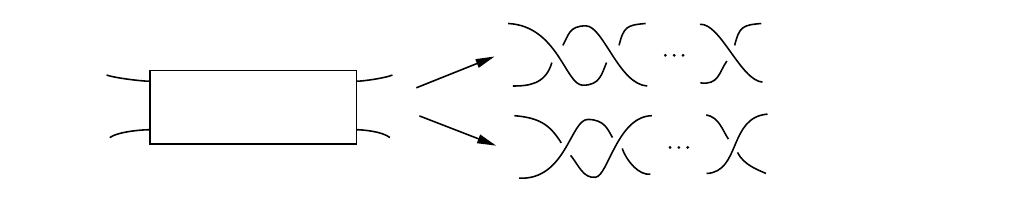}
\begin{picture}(400,0)(0,0)
\put(102,52){$k$}
\put(160,77){$k \geq 0$}
\put(160,30){$k < 0$}
\end{picture}
    \caption{$k$ half-twists.}
    \label{fig:twist}
\end{figure}
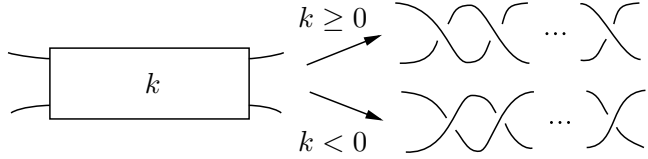

Recall that for an integer $q$, the (unoriented) link $T(2,q)$ shown in Figure~\ref{fig:torus_link} is called a \emph{$(2,q)$-torus link}. 
We refer to the diagram in Figure~\ref{fig:torus_link} as the \emph{standard diagram} of $T(2, q)$.

\begin{figure}[htbp]
\centering\includegraphics[width=1.2\textwidth]{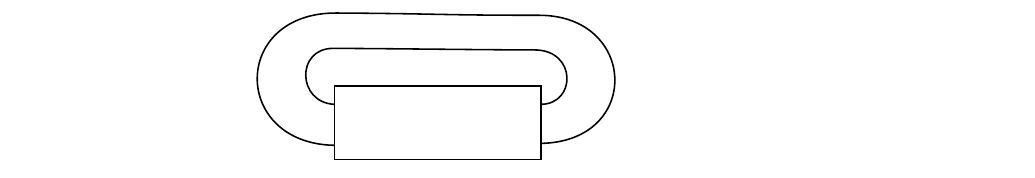}
\begin{picture}(400,0)(0,0)
\put(200,37){$q$}
\end{picture}
    \caption{The standard diagram of the $(2,q)$-torus link $T(2, q)$.}
    \label{fig:torus_link}
\end{figure}
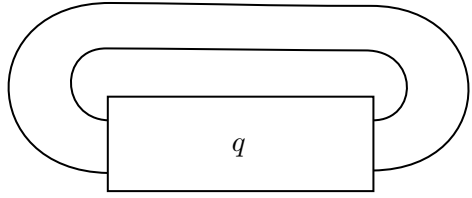

\begin{proposition}
\label{prop: torus_link}
Let $R$ be an arbitrary commutative ring, $(A, B)$ a tribracket bracket with respect to $X_2$ and $R$, and $q$ an arbitrary integer. Then, the quantum enhancement polynomial $\Phi^{(A, B)}{X_2}(T(2, q))$ of an oriented $T(2,q)$ is given as follows, 
where we set $n=1$ if $q=0$ or the sign of every crossing of the standard diagram of oriented $T(2,q)$ is positive, and $n=-1$ otherwise.  
    \begin{enumerate}
        \item If $q$ is odd, then $\Phi^{(A, B)}_{X_2}(T(2, q))$ does not depend of orientations and we have 
            \[
                \Phi^{(A, B)}_{X_2}(T(2, q)) = \sum_{a, b \in {X_2}}u^{w^{-q}(A_{a,b,b}^{n|q|}\delta^2 + (A_{a,b,b}^n+B_{a,b,b}^n\delta)^{|q|} - A_{a,b,b}^{n|q|})}. 
            \]

        \item If $q$ is even, then we have 
            \[
                 \Phi^{(A, B)}_{X_2} (T(2,q)) = \sum_{a,b,c \in X_2} u^{\alpha}, 
            \]
where we set 
            \begin{align*}
                \alpha &= w^{-q} (A_{a,b,c}^{n\frac{|q|}{2}}A_{a,c,b}^{n\frac{|q|}{2}}\delta^2 +(A_{a,b,c}^n+B_{a,b,c}^n\delta)^{\frac{|q|}{2}}(A_{a,c,b}^n+B_{a,c,b}^n\delta)^{\frac{|q|}{2}} - A_{a,b,c}^\frac{|q|}{2}A_{a,c,b}^\frac{|q|}{2})
            \end{align*}
if $nq \geq 0$, and we set 
            \begin{align*}
                \alpha &= w^{q} (B_{a,b,c}^{n\frac{|q|}{2}}B_{a,c,b}^{n\frac{|q|}{2}}\delta^2 +(B_{a,b,c}^n+A_{a,b,c}^n\delta)^{\frac{|q|}{2}}(B_{a,c,b}^n+A_{a,c,b}^n\delta)^{\frac{|q|}{2}} - B_{a,b,c}^\frac{|q|}{2}B_{a,c,b}^\frac{|q|}{2})
            \end{align*}
            otherwise. 
    \end{enumerate}
\end{proposition}
The proof is provided in the Appendix.

Let $q$ be an integer. 
The (unoriented) knot $\mathit{TW}(q)$ shown in Figure~\ref{fig:twist_knot} is called the \emph{$q$-twist knot}. 
We refer to the diagram in Figure~\ref{fig:twist_knot} as the \emph{standard diagram} of $\mathit{TW}(q)$.

\begin{figure}
\centering\includegraphics[width=1.2\textwidth]{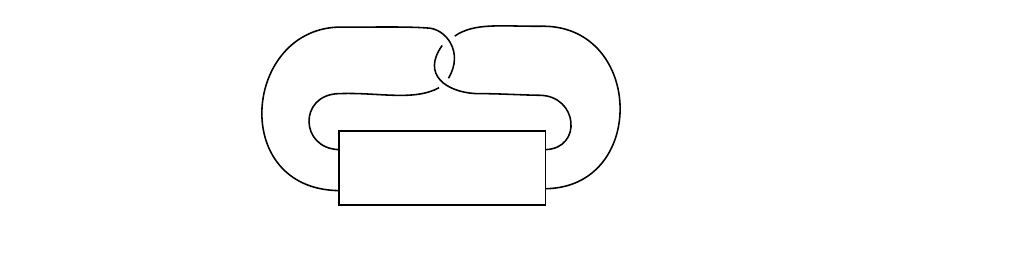}
\begin{picture}(400,0)(0,0)
\put(202,52){$q$}
\end{picture}
    \caption{The standard diagram of the $q$-twist knot $\mathit{TW}(q)$. }
    \label{fig:twist_knot}
\end{figure}
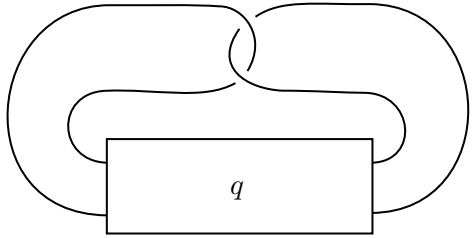

\begin{proposition}
\label{prop: twist_knot}
Let $R$ be an arbitrary commutative ring, $(A, B)$ a tribracket bracket with respect to $X_2$ and $R$, and $q$ an arbitrary integer. 
Then, the quantum enhancement polynomial $\Phi^{(A, B)}_{X_2}(\mathit{TW}(q))$ of an oriented $\mathit{TW}(q)$ is independent of orientation and is given as follows, 
where we set $n=1$ $($resp. $m=1$$)$ if $q \geq 0$ $($resp. $q < 0$$)$. 
    \[
        \Phi^{(A, B)}_{X_2}(\mathit{TW}(q)) = \sum_{a, b \in X_2} u^\alpha, 
    \]
   where if $q$ is even, $\alpha$ is defined by 
    \[
        \alpha =w^{n(q+2)}((A_{a,b,b}^n\delta + B_{a,b,b}^n)^{|q|} B_{a,b,b}^{n|q|}(\delta^2-1))(A_{a,b,b}^{2n}\delta + 2A_{a,b,b}^nB_{a,b,b}^n + B_{a,b,b}^{2n}\delta) - B_{a,b,b}^{n(|q|+2)}\delta(\delta^2-1)),
    \]
    and if $q$ is odd, $\alpha$ is defined by 
    \[
        \alpha = w^{n(q+2)}((A_{a,b,b}^n + B_{a,b,b}^n\delta)^{|q|} A_{a,b,b}^{n|q|}(\delta^2-1))(A_{a,b,b}^{2n}\delta + 2A_{a,b,b}^nB_{a,b,b}^n + B_{a,b,b}^{2n}\delta) - A_{a,b,b}^{n(|q|+2)}\delta(\delta^2-1)). 
    \]
\end{proposition}
The proof is provided in the Appendix.

\section{Basic properties of quantum enhancement polynomials}
\label{sec: Basic properties of quantum enhancement polynomials}

In this section, we discuss the basic properties of quantum enhancement polynomials associated with 
the canonical two-element tribracket $X_2$.

\subsection{$X_2$-colorings of link diagrams}
\label{subsec:X-colorings of link diagrams}

Let $L = K_1 \cup \cdots \cup K_n$ be an $n$-component link, and let $D_L \subset \mathbb{R}^2$ be a diagram of $L$. For each $i \in \{1, \ldots, n\}$, let $D_i \subset D_L$ be the diagram of $K_i$. A pair $(\mathcal{D}_1, \mathcal{D}_2)$ of (possibly empty) subsets $\mathcal{D}_1, \mathcal{D}_2$ of $\{D_1, \ldots, D_n\}$ with $\mathcal{D}_1 \cap \mathcal{D}_2 = \emptyset$ and $\mathcal{D}_1 \cup \mathcal{D}_2 = \{D_1, \ldots, D_n\}$ is called a \emph{decomposition} of $\{D_1, \ldots, D_n\}$.

For a decomposition $(\mathcal{D}_1, \mathcal{D}_2)$ of $\{D_1, \ldots, D_n\}$, we denote by $C(a, \mathcal{D}_1, \mathcal{D}_2)$ the $X_2$-coloring of $D_L$ satisfying the following two conditions:

\begin{enumerate} 
\item The unbounded region of $D_L$ is colored by $a$. 
\item If two regions $r, r' \in \mathcal{R}(D_L)$ are adjacent across $D_i \in \mathcal{D}_1$, 
then $r$ and $r'$ are colored by different elements of $X_2$. If $r, r' \in \mathcal{R}(D_L)$ 
are adjacent across $D_i \in \mathcal{D}_2$, then $r$ and $r'$ are colored by the same element of $X_2$. \end{enumerate}
Figure~\ref{fig:coloring_example} shows the coloring $C(0, \{D_1\}, \{D_2, D_3\})$ of a diagram of a $3$-component link. 
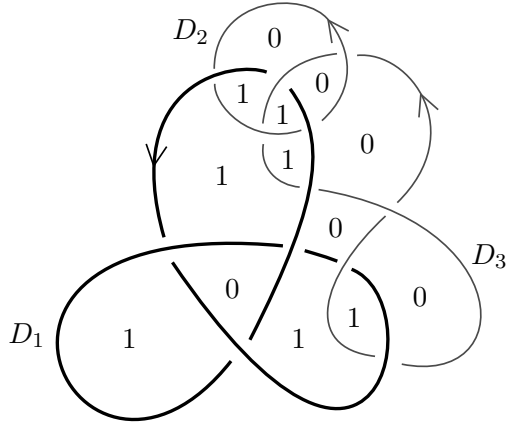
\begin{figure}[htbp]
\centering\includegraphics[width=1.0\textwidth]{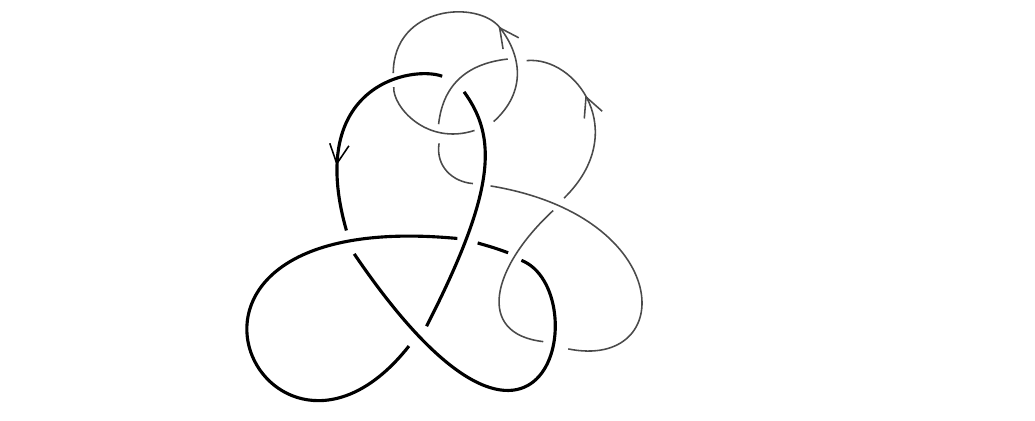}
\begin{picture}(400,0)(0,0)

\put(72,50){$D_1$}
\put(134,165){$D_2$}
\put(247,80){$D_3$}

\put(225,64){$0$}
\put(154,67){$0$}
\put(193,90){$0$}
\put(205,122){$0$}
\put(188,145){$0$}
\put(170,162){$0$}

\put(200,56){$1$}
\put(179,48){$1$}
\put(115,48){$1$}
\put(150,110){$1$}
\put(175,116){$1$}
\put(173,133){$1$}
\put(158,141){$1$}

\end{picture}
    \caption{The coloring $C(0, \{D_1\}, \{D_2, D_3\})$.}
    \label{fig:coloring_example}
\end{figure}

\begin{proposition}\label{prop:coloring_number}
Let $D_L = D_1 \cup \cdots \cup D_n$ be a diagram of a link $L = K_1 \cup \cdots \cup K_n$. 
Then, a map $C : \mathcal{R}(D_L) \to X_2$ is an $X_2$-coloring of $D_L$ if and only if 
there exists a decomposition $(\mathcal{D}_1, \mathcal{D}_2)$ of $\{D_1, \ldots, D_n\}$ and an element $a \in X_2$ 
such that $C = C(a, \mathcal{D}_1, \mathcal{D}_2)$.
\end{proposition}

\begin{proof}
We first prove the "if" part. Let $(\mathcal{D}_1, \mathcal{D}_2)$ be a decomposition of 
$\{D_1, \ldots, D_n\}$ and $a \in X_2$. We will show that $C(a, \mathcal{D}_1, \mathcal{D}_2)$ is an $X_2$-coloring of $D_L$. 
Let $c$ be a crossing of $D_i$ and $D_j$.

If $D_i, D_j \in \mathcal{D}_1$, then by the definition of $C(a, \mathcal{D}_1, \mathcal{D}_2)$, the regions around $c$ are colored as shown in either case in Figure~\ref{fig:crossing_case1} up to over/under information and orientations. 
    \begin{figure}[htpb]
    \centering\includegraphics[width=1.2\textwidth]{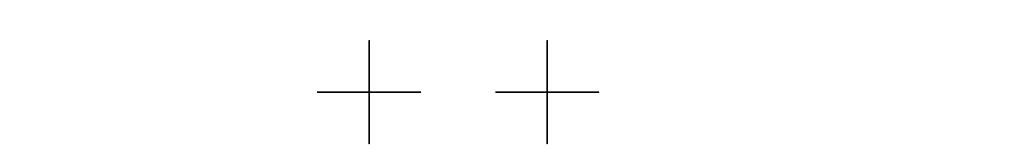}
\begin{picture}(400,0)(0,0)

\put(130,44){$D_i$}
\put(164,79){$D_j$}
\put(155,57){$0$}
\put(179,57){$1$}
\put(155,32){$1$}
\put(179,32){$0$}

\put(217,44){$D_i$}
\put(251,79){$D_j$}
\put(242,57){$1$}
\put(266,57){$0$}
\put(242,32){$0$}
\put(266,32){$1$}

\end{picture}
        \caption{The case where $D_i, D_j \in \mathcal{D}_1$.}
        \label{fig:crossing_case1}
    \end{figure}
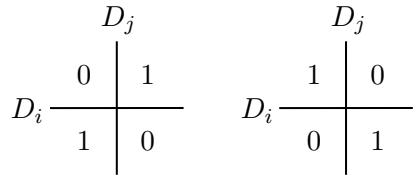
If $D_i, D_j \in \mathcal{D}_2$, then the regions around $c$ are colored as shown in either case in Figure~\ref{fig:crossing_case2} up to over/under information and orientations. 
    \begin{figure}[htpb]
        \centering\includegraphics[width=1.2\textwidth]{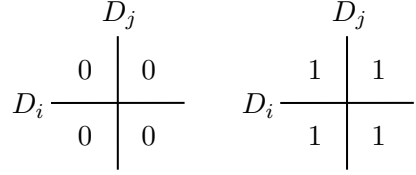}
\begin{picture}(400,0)(0,0)

\put(130,44){$D_i$}
\put(164,79){$D_j$}
\put(155,57){$0$}
\put(179,57){$0$}
\put(155,32){$0$}
\put(179,32){$0$}

\put(217,44){$D_i$}
\put(251,79){$D_j$}
\put(242,57){$1$}
\put(266,57){$1$}
\put(242,32){$1$}
\put(266,32){$1$}

\end{picture}
        \caption{The case where $D_i, D_j \in \mathcal{D}_2$.}
        \label{fig:crossing_case2}
    \end{figure}
Finally, if $D_i \in \mathcal{D}_1$ and $D_j \in \mathcal{D}_2$, or if $D_j \in \mathcal{D}_1$ and $D_i \in \mathcal{D}_2$, then the regions around $c$ are colored as in either case in Figure~\ref{fig:crossing_case3} up to over/under information and orientations. 
    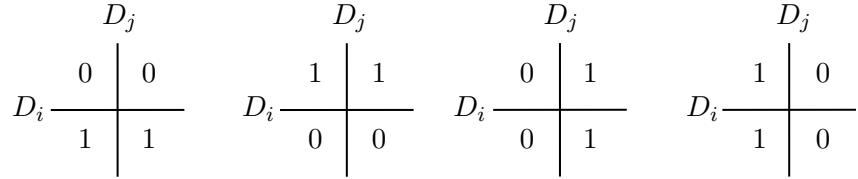
\begin{figure}[htpb]
        \centering\includegraphics[width=1.2\textwidth]{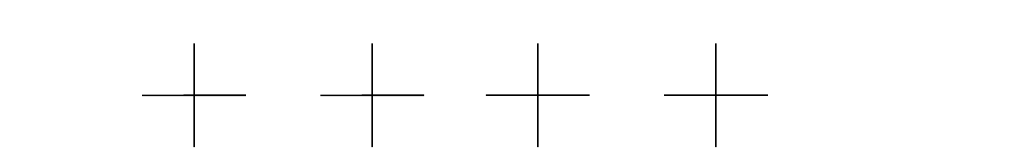}
\begin{picture}(400,0)(0,0)

\put(45,44){$D_i$}
\put(79,79){$D_j$}
\put(70,57){$0$}
\put(94,57){$0$}
\put(70,32){$1$}
\put(94,32){$1$}

\put(132,44){$D_i$}
\put(166,79){$D_j$}
\put(157,57){$1$}
\put(181,57){$1$}
\put(157,32){$0$}
\put(181,32){$0$}

\put(212,44){$D_i$}
\put(246,79){$D_j$}
\put(237,57){$0$}
\put(261,57){$1$}
\put(237,32){$0$}
\put(261,32){$1$}

\put(300,44){$D_i$}
\put(334,79){$D_j$}
\put(325,57){$1$}
\put(349,57){$0$}
\put(325,32){$1$}
\put(349,32){$0$}

\end{picture}
        \caption{The case where $D_i \in \mathcal{D}_1$ and $D_j \in \mathcal{D}_2$, or $D_j \in \mathcal{D}_1$ and $D_i \in \mathcal{D}_2$.} 
        \label{fig:crossing_case3}
    \end{figure}
Therefore, in all cases, $C(a, \mathcal{D}_1, \mathcal{D}_2)$ satisfies the relation shown in Figure~\ref{fig:1} around the crossing $c$, 
and hence $C(a, \mathcal{D}_1, \mathcal{D}_2)$ is an $X_2$-coloring of $D_L$. 

Next, we prove the "only if" part. 
Let $C$ be an $X_2$-coloring of $D_L$. 
We call a connected component of $D_L$ minus its crossings an \emph{edge} of $D_L$. 
For each edge $\alpha$ of $D_L$, we denote by $r_\alpha, r'_\alpha \in \mathcal{R}(D_L)$ the two regions facing each other on $\alpha$. 

\begin{claim}\label{claim:same_color}
If there exists an edge $\alpha \subset D_i$ such that $C(r_\alpha) = C(r'_\alpha)$, then for any edge $\beta \subset D_i$, we have $C(r_\beta) = C(r'_\beta)$. 
    \end{claim}
    \begin{proof}
Suppose that there exists an edge $\alpha \subset D_i$ of $D_L$ with $C(r_\alpha) = C(r'_\alpha) = x$. Let $c$ be the endpoint of $\alpha$, and let $\beta \subset D_i$ be the adjacent edge across $c$. Since $C$ is an $X_2$-coloring, we have $C(r_\beta) = C(r'_\beta) = x$ or $C(r_\beta) = C(r'_\beta) = x+1$ (See Figure~\ref{fig:coloring_same_color}). 
Applying the same argument repeatedly, we get the desired result.
\begin{figure}[htbp]
\centering\includegraphics[width=1.2\textwidth]{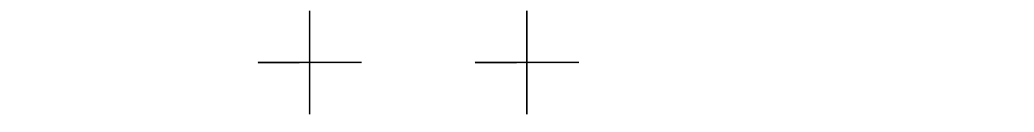}
\begin{picture}(400,0)(0,0)

\put(104,44){$\alpha$}
\put(168,44){$\beta$}
\put(145,38){$c$}
\put(126,60){$x$}
\put(150,60){$x$}
\put(126,28){$x$}
\put(150,28){$x$}

\put(212,44){$\alpha$}
\put(276,44){$\beta$}
\put(250,38){$c$}
\put(230,60){$x$}
\put(254,60){$x+1$}
\put(230,28){$x$}
\put(254,28){$x+1$}

\end{picture}
\caption{Colors of the regions around $c$.}
\label{fig:coloring_same_color}
\end{figure}
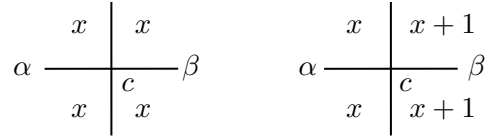
    \end{proof}
 Let $\mathcal{D}_2$ be the set of all $D_i$ containing an edge $\alpha$ of $D_L$ such that $C(r_\alpha) = C(r'_\alpha)$. Set $\mathcal{D}_1 = \{D_1, \ldots, D_n\} - \mathcal{D}_2$. Then, by Claim~\ref{claim:same_color}, for any $D_i \in \mathcal{D}_2$, 
 the regions on both sides of any edge in $D_i$ are colored the same. 
 Moreover, for any $D_i \in \mathcal{D}_1$, the regions on both sides of any edge in $D_i$ are colored differently. 
 Therefore, we have $C = C(a, \mathcal{D}_1, \mathcal{D}_2)$, where $a$ is the color of the unbounded region of $D_L$.
  \end{proof}

Using Proposition~\ref{prop:coloring_number}, we can show that 
the number of $X_2$-colorings of a link diagram depends only on the number of components of the link, as follows.
\begin{corollary}
\label{cor:knot_coloring}
Let $D_L$ be a diagram of an $n$-component link $L$. Then, the number of $X_2$-colorings of $D_L$ is $2^{n+1}$. 
In particular, if $L$ is a knot, then the number of $X_2$-colorings of $D_L$ is four, consisting of two trivial $X_2$-colorings and two checkerboard $X_2$-colorings.

\end{corollary}
\begin{proof}
Let $L = K_1 \cup \cdots \cup K_n$, and let $D_i \subset D_L$ be the diagram of $K_i$. By Proposition~\ref{prop:coloring_number},
the set of $X_2$-colorings of $D_L$ corresponds one-to-one with the set of pairs of a subset $\mathcal{D}_1 \subset \{D_1, \ldots, D_n\}$ 
and an element $a \in X_2$, so the result follows immediately.

Now, suppose $n=1$, i.e., $L$ is a knot. 
If $\mathcal{D}_1 = \emptyset$, then $C(a, \emptyset, {D_1})$ is a trivial $X_2$-coloring. 
If $\mathcal{D}_1 = {K_1}$, then $C(a, {D_1}, \emptyset)$ is a checkerboard $X_2$-coloring. 
\end{proof}

For a diagram $D_L$ of a link $L$, let $|D_L|$ denote the image of an immersion of circles obtained by forgetting the over/under crossing information in $D_L$.

\begin{lemma}\label{lem:coloring_pair}
Let $D_L$ be a diagram of a link $L$, and let $C$ be an $X_2$-coloring of $D_L$. 
Let $r_0, r_1, l_0, l_1$ denote the number of crossings in $|D_L|$ around which the regions are colored as in Figure~$\ref{fig:coloring_pair}$, respectively, 
where in the figure, the local models of colorings are ordered from left to right. Then, we have $r_0 = l_1$ and $r_1 = l_0$.

    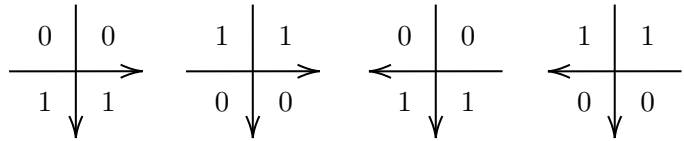
\begin{figure}[htpb]
        \centering\includegraphics[width=1.2\textwidth]{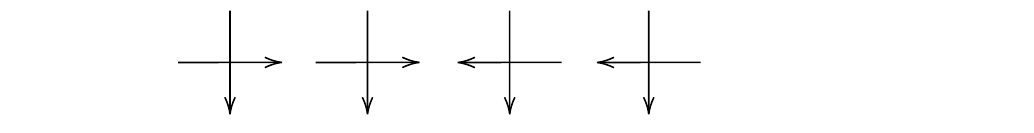}
\begin{picture}(400,0)(0,0)

\put(88,57){$0$}
\put(112,57){$0$}
\put(88,32){$1$}
\put(112,32){$1$}

\put(155,57){$1$}
\put(179,57){$1$}
\put(155,32){$0$}
\put(179,32){$0$}

\put(224,57){$0$}
\put(248,57){$0$}
\put(224,32){$1$}
\put(248,32){$1$}

\put(292,57){$1$}
\put(316,57){$1$}
\put(292,32){$0$}
\put(316,32){$0$}

\end{picture}
        \caption{Colors of the regions around vertices of $|D_L|$.}
        \label{fig:coloring_pair}
    \end{figure}
\end{lemma}
\begin{proof}
By Proposition~\ref{prop:coloring_number},  
 there exists a decomposition $(\mathcal{D}_1, \mathcal{D}_2)$ of $\{D_1, \ldots, D_n\}$ and an element $a \in X_2$ such that $C = C(a, \mathcal{D}_1, \mathcal{D}_2)$. The horizontal segments in Figure~\ref{fig:coloring_pair} belong to $\mathcal{D}_1$, while the vertical segments belong to $\mathcal{D}_2$. The number of times that segments of $\mathcal{D}_1$ cross segments of $\mathcal{D}_2$ from right to left and from left to right are the same. Thus, we have $r_0 + r_1 = l_0 + l_1$.

On the other hand, the total number of times the colors of the regions on both sides of segments of $\mathcal{D}_2$ change from 0 to 1 and from 1 to 0 at crossings with $\mathcal{D}_1$ are the same. Thus, we have $r_0 + l_0 = r_1 + l_1$. The assertion follows from these two equations. 
\end{proof}

\subsection{Complementary colorings}
\label{subsec:Complementary colorings}

Let $L$ be a link, $D_L$ a diagram of $L$, and $C : \mathcal{R}(D_L) \to X_2$ an $X_2$-coloring of $D_L$. 
Define a map $\overline{C}: \mathcal{R}(D_L) \to X_2$ by $\overline{C}(r) = C(r) + 1$ for each region $r$ of $D_L$. 
Since we have $[a+1, b+1, c+1] = a + b + c + 3 = [a, b, c] + 1$ for any $a, b, c \in X_2$, 
the map $\overline{C}$ thus defined is also an $X_2$-coloring of $D_L$. 
We call this coloring $\overline{C}$ the \emph{complementary coloring} of $C$.

In this section, we show that regardless of the commutative ring $R$ and 
the tribracket bracket $(A, B)$ with respect to $X_2$ and $R$, 
the quantum enhancement polynomial $\beta^{(A, B)}_{X_2}(D_L, C)$ 
coincides with $\beta^{(A, B)}_{X_2}(D_L, \overline{C})$.

\begin{lemma}\label{lem:tri_eq}
Let $R$ be a commutative ring, and let $(A, B)$ be a tribracket bracket with respect to $X_2$ and $R$. 
Then, the following equalities hold: 
    \begin{enumerate}
        \item $A_{0,0,0} = A_{1,1,1}$ \label{tri_eq1}
        \item $B_{0,0,0} = B_{1,1,1}$ \label{tri_eq2}
        \item $A_{0,1,1} = A_{1,0,0}$ \label{tri_eq3}
        \item $B_{0,1,1} = B_{1,0,0} = B_{0,0,0}A_{0,0,0}^{-2}A_{0,1,1}^2$ \label{tri_eq4}
        \item $A_{0,0,1}A_{0,1,0} = A_{1,1,0}A_{1,0,1}$ \label{tri_eq5}
        \item $B_{0,0,1}A_{0,1,0} = B_{1,1,0}A_{1,0,1}$ \label{tri_eq6}
        \item $A_{0,1,0}B_{1,0,1} = A_{1,0,1}B_{0,1,0}$ \label{tri_eq7}
        \item $A_{0,0,1}B_{0,1,0} = A_{1,1,0}B_{1,0,1}$ \label{tri_eq8}
        \item $B_{0,0,1}B_{0,1,0} = B_{1,1,0}B_{1,0,1}$ \label{tri_eq9}
        \item $A_{0,0,1}B_{1,1,0} = A_{1,1,0}B_{0,0,1}$ \label{tri_eq10}
    \end{enumerate}
\end{lemma}
\begin{proof}
We prove each statement of the lemma in turn: 
  \begin{enumerate}
    \item        
    Substituting $(a, b, c, d) = (0,0,0,1)$ into equation (\ref{3a})  in Definition~\ref{def:tribracketbracket},   
    we obtain $A_{0,0,0}A_{0,0,1}A_{0,0,1} = A_{0,0,1}A_{0,0,1}A_{1,1,1}$. 
     Hence, we have $A_{0,0,0} = A_{1,1,1}$. 
\item Substituting $(a, b, c, d) = (1,0,1,1)$ into equation (\ref{3b}) in Definition~\ref{def:tribracketbracket}, we obtain $A_{1,0,1}B_{1,1,1} = B_{0,0,0}A_{1,0,1}$. Thus, we have $B_{0,0,0} = B_{1,1,1}$.
    \item 
    Substituting $(a, b, c, d) = (0,1,1,1)$ into equation (\ref{3a}), we get $A_{0,1,1} A_{1,0,0} = A_{1,0,0} A_{0,1,1}$. Thus, we conclude that $A_{0,1,1} = A_{1,0,0}$.
    \item 
    By the condition (\ref{2}) in Definition~\ref{def:tribracketbracket}, we have $-A_{0,0,0}^2 B_{0,0,0}^{-1} = -A_{0,1,1}^2 B_{0,1,1}^{-1}$. 
    Therefore, $B_{0,1,1} = A_{0,0,0}^{-2} B_{0,0,0} A_{0,1,1}^2$. 
    Similarly, from the same equation, we get $B_{1,0,0} = A_{0,0,0}^{-2}B_{0,0,0}A_{1,0,0}^2$. 
    Since $A_{0,1,1} = A_{1,0,0}$ by (\ref{tri_eq3}), we have $B_{0,1,1} = B_{1,0,0}$.
        \item 
    Substituting $(a, b, c, d) = (0,1,0,1)$ into equation (\ref{3a}) in Definition~\ref{def:tribracketbracket},  
    we get $A_{0,1,0}A_{0,1,1}A_{0,0,1} = A_{1,1,0}A_{0,1,1}A_{1,0,1}$. 
     Thus, we have $A_{0,0,1}A_{0,1,0} = A_{1,1,0}A_{1,0,1}$. 
    \item 
    Substituting $(a, b, c, d) = (0,0,1,1)$ into equation (\ref{3b}) in Definition~\ref{def:tribracketbracket},  
    we get $B_{0,1,0}B_{0,1,1}A_{0,0,1} = A_{1,1,0}B_{0,1,1}B_{1,0,1}$. 
    Since $A_{1,1,0} = A_{0,0,1}A_{0,1,0}A_{1,0,1}^{-1}$ by (\ref{tri_eq5}), we have $B_{0,0,1}A_{0,1,0} = B_{1,1,0}A_{1,0,1}$. 
    \item 
    Substituting $(a, b, c, d) = (0,1,0,1)$ into equation (\ref{3c}) in Definition~\ref{def:tribracketbracket},  
    we get $A_{0,1,0}A_{0,1,1}A_{0,0,1} = A_{1,1,0}A_{0,1,1}A_{1,0,1}$. 
   Since $A_{1,1,0} = A_{0,0,1}A_{0,1,0}A_{1,0,1}^{-1}$ by (\ref{tri_eq5}), we have $A_{0,1,0}B_{1,0,1} = A_{1,0,1}B_{0,1,0}$. 
    \item 
    By equation (\ref{tri_eq5}), 
    we have $A_{0,1,0} = A_{1,1,0}A_{1,0,1}A_{0,0,1}^{-1}$. 
    Substituting this into equation (\ref{tri_eq7}), we get $A_{0,0,1}B_{0,1,0} = A_{1,1,0}B_{1,0,1}$. 
    \item 
    By equation (\ref{tri_eq7}), 
    we have $A_{0,0,1}B_{0,1,0} = A_{1,1,0}B_{1,0,1}$. 
    Substituting this into equation (\ref{tri_eq7}), we get $B_{0,0,1}B_{0,1,0} = B_{1,1,0}B_{1,0,1}$. 
    \item 
    By equation (\ref{tri_eq5}), 
    we have $A_{0,1,0} = A_{1,1,0}A_{1,0,1}A_{0,0,1}^{-1}$. 
    Substituting this into equation (\ref{tri_eq6}), we get $A_{0,0,1}B_{1,1,0} = A_{1,1,0}B_{0,0,1}$. 
    \end{enumerate}
\end{proof}

\begin{proposition}\label{prop:comp_color}
Let $R$ be a commutative ring, and let $(A, B)$ be a tribracket bracket with respect to $X_2$ and $R$. 
Let $D_L$ be a diagram of a link $L$. 
Then, for any $X_2$-coloring $C$ of $D_L$, we have
$\beta^{(A, B)}_{X_2}(D_L, C) = \beta^{(A, B)}_{X_2}(D_L, \overline{C})$. 
\end{proposition}

\begin{proof}
We will show that for any state $s : \mathcal{C}(D_L) \to \{A, B\}$ of $D_L$, the equality
\[\prod_{c \in \mathcal{C}(D_L)} s(c)(\cl_C(c))^{\sign (c)} = \prod_{c \in \mathcal{C}(D_L)} s(c)(\cl_{\overline{C}}(c))^{\sign(c)}\] 
holds. 
Let $D_L = D_1 \cup \cdots \cup D_n$. 
By Proposition~\ref{prop:coloring_number},  
there exists a decomposition $(\mathcal{D}_1, \mathcal{D}_2)$ of $\{D_1, \ldots, D_n\}$ and an element $a \in X_2$ such that $C = C(a, \mathcal{D}1, \mathcal{D}2)$. Let $\mathcal{C}(D_L) = \{c_1, \ldots, c_m\}$, where each $c_i$ is a crossing of $D_{i_1}$ and $D_{i_2}$.

Let $m'$ denote the number of crossings in $|D_L|$ where the regions are colored as shown in Figure~~\ref{fig:coloring_pair}. 
By Lemma~\ref{lem:coloring_pair}, we can assume (by swapping the subscripts if necessary) that crossings of $D_L$ satisfy the following conditions: 
\begin{enumerate}\renewcommand{\labelenumi}{(\ref{prop:comp_color}-\alph{enumi})}
\item 
For any odd $i\ (1 \leq i \leq m'-1)$, 
the regions around $c_i$ and $c_{i+1}$ are colored as in either of the models shown Figure~\ref{fig:ci_ci1}, 
up to over/under information. 
    \begin{figure}[htpb]
        \centering\includegraphics[width=1.2\textwidth]{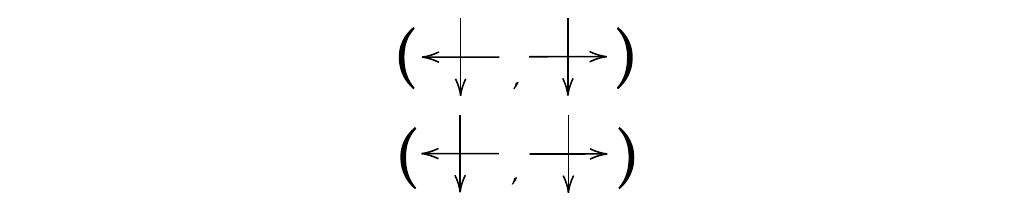}
\begin{picture}(400,0)(0,0)

\put(120,87){$(c_i, c_{i+1}) = $}

\put(200,100){$0$}
\put(224,100){$0$}
\put(200,75){$1$}
\put(224,75){$1$}

\put(253,100){$1$}
\put(277,100){$1$}
\put(253,75){$0$}
\put(277,75){$0$}

\put(120,42){$(c_i, c_{i+1}) = $}

\put(200,55){$0$}
\put(224,55){$0$}
\put(200,30){$1$}
\put(224,30){$1$}

\put(253,55){$1$}
\put(277,55){$1$}
\put(253,30){$0$}
\put(277,30){$0$}

\end{picture}
        \caption{Colors of the regions around $c_i$ and $c_{i+1}$.}
        \label{fig:ci_ci1}
    \end{figure}
\item 
For any integer $i\ (m'+1 \leq i \leq m)$, the regions around $c_i$ are colored as in either of the models shown Figures~\ref{fig:crossing_case1} and \ref{fig:crossing_case2} up to over/under information. 
    \end{enumerate}
We first show that the equality
$s(c_i)(\cl_C(c_i))^{\sign(c_i)} = s(c_i)(\cl_{\overline{C}}(c_i))^{\sign(c_i)}$ 
holds for any crossing $c_i \ (m'+1 \leq i \leq m)$. 
For a crossing $c_i \ (m'+1 \leq i \leq m)$, if $D_{i_1}, D_{i_2} \in \mathcal{D}_1$, 
the colors of regions around $c_i$ by $C$ or $C'$ is as in Figure~\ref{fig:crossing_case1}. 
Since $A_{0,1,1} = A_{1,0,0}$ and $B_{0,1,1} = B_{1,0,0}$ by Lemma~\ref{lem:tri_eq}, we have  
$s(c_i)(\cl_C(c_i))^{\sign(c_i)} = s(c_i)(\cl_{\overline{C}}(c_i))^{\sign(c_i)}$. 
If $D_{i_1}, D_{i_2} \in \mathcal{D}_2$, 
the colors of regions around $c_i$ by $C$ or $C'$ is as in Figure~\ref{fig:crossing_case2} . 
Since $A_{0,0,0} = A_{1,1,1}$ and $B_{0,0,0} = B_{1,1,1}$ by Lemma~\ref{lem:tri_eq}, we have 
$s(c_i)(\cl_C(c_i))^{\sign(c_i)} = s(c_i)(\cl_{\overline{C}}(c_i))^{\sign(c_i)}$. 

Next, we show that the equality 
\begin{align*}
s(c_{i})(\cl_C(c_i))^{\sign(c_i)}s(c_{i+1})(\cl_C(c_{i+1}))^{\sign(c_{i+1})} = s(c_i)(\cl_{\overline{C}}(c_i))^{\sign(c_i)}s(c_{i+1})(\cl_{\overline{C}}(c_{i+1}))^{\sign(c_{i+1})}
\end{align*} 
holds for any odd $i\ (1 \leq i \leq m'-1)$. 
Assume the colors of regions around $c_i, c_{i+1}$ are as in the top of Figure~\ref{fig:ci_ci1}. 
The same proof applies to the case where the regions are colored as in the bottom of Figure~\ref{fig:ci_ci1}. 
    \begin{itemize}
        \item 
        If $\sign(c_i) = \sign(c_{i+1}) = 1$, 
        then by equations (\ref{lem:tri_eq}-\ref{tri_eq5}), (\ref{lem:tri_eq}-\ref{tri_eq6}), (\ref{lem:tri_eq}-\ref{tri_eq8}), (\ref{lem:tri_eq}-\ref{tri_eq9}) in 
        Lemma~\ref{lem:tri_eq}, we have 
        \[s(c_i)(0,0,1)s(c_{i+1})(0,1,0) = s(c_i)(1,1,0)s(c_{i+1})(1,0,1)\] 
        regardless of the values of $s(c_i), s(c_{i+1})$, 
        which implies the assertion. 
        \item 
        If $\sign(c_i) = 1$ and $\sign(c_{i+1}) = -1$:  
            \begin{itemize}
                \item If $s(c_i) = s(c_{i+1})$, we have 
                \[ s(c_i)(0,0,1)s(c_{i+1})(0,0,1)^{-1} = 1 =s(c_i)(1,1,0)s(c_{i+1})(1,1,0)^{-1},\] 
                so the equality holds. 
                \item If $s(c_i) = A$ and $s(c_{i+1}) = B$, it follows from the equation (\ref{lem:tri_eq}-\ref{tri_eq10}) in Lemma~\ref{lem:tri_eq} that 
                \begin{align*}
                s(c_i)(0,0,1)s(c_{i+1})(0,0,1)^{-1} &= A_{0,0,1}B_{0,0,1}^{-1} = A_{1,1,0}B_{1,1,0}^{-1} \\
                &=s(c_i)(1,1,0)s(c_{i+1})(1,1,0)^{-1},
                \end{align*} 
                which proves the assertion. The same argument applies when $s(c_i) = B$ and $s(c_{i+1}) = A$. 
            \end{itemize}
        \item If $\sign(c_i) = -1$ and $\sign(c_{i+1}) = 1$, the proof is analogous to the case where $\sign(c_i) = 1$ and $\sign(c_{i+1}) = -1$, 
      using equation (\ref{tri_eq7}) instead of (\ref{tri_eq10}). 
        \item If $\sign(c_i) = \sign(c_{i+1}) = -1$,  
        the proof is analogous to the case where $\operatorname{sign}(c_i) = \operatorname{sign}(c_{i+1}) = 1$.
    \end{itemize}
Thus, we obtain
    \begin{align*}
    &\prod_{i=1}^m s(c_i)(\cl_C(c_i))^{\sign(c_i)}\\ 
    &= \left( \prod_{i \in \{1,3,\ldots, m'-1\}} s(c_i)(\cl_C(c_i))^{\sign(c_i)}s(c_{i+1})(\cl_C(c_{i+1}))^{\sign(c_{i+1})}\right) \left(\prod_{i=m'+1}^m s(c_i)(\cl_C(c_i))^{\sign(c_i)}\right)\\
    &= \left(\prod_{i \in \{1,3,\ldots, m'-1\}} s(c_i)(\cl_{\overline{C}}(c_i))^{\sign(c_i)}s(c_{i+1})(\cl_{\overline{C}}(c_{i+1}))^{\sign(c_{i+1})}\right) \left(\prod_{i=m'+1}^m s(c_i)(\cl_{\overline{C}}(c_i))^{\sign(c_i)}\right)\\
    &= \prod_{i=1}^m s(c_i)(\cl_{\overline{C}}(c_i))^{\sign(c_i)} . 
    \end{align*}
Therefore, we conclude that $\beta^{(A, B)}{X_2}(D_L, C) = \beta^{(A, B)}{X_2}(D_L, \overline{C})$. \
\end{proof}

\subsection{Mirror images}
Let $D_L$ be a diagram of a link $L$. 
Let $L^*$ denotes the mirror image of $L$. 
By $D^*_L$ we denote the diagram of $L^*$ obtained by changing all the crossings of $D_L$. 
Let $R$ be a commutative ring, and $(A, B)$ a tribracket bracket with respect to 
$X_2$ and $R$. 
Define maps $\overline{A}, \overline{B} : X_2^3 \to R^\times$ by 
$\overline{A}(a, b, c) = A_{a, b, c}^{-1}, \overline{B}(a, b, c) = B_{a, b, c}^{-1}$ for each $a, b, c \in X_2$. 

It can be verified that $(\overline{A}, \overline{B})$ satisfies the conditions~\ref{1} and \ref{2} of Definition~\ref{def:tribracketbracket} as follows.
Let $\delta \in R$ be the element satisfying $\delta = - A_{a, b, c}B_{a, b, c}^{-1} - A_{a, b, c}^{-1}B_{a, b, c}$ for any $a, b, c \in X_2$, 
and let $w$ be the distinguished element of $(A, B)$. 
Then, for any $a, b, c \in X_2$, we have 
$- \overline{A}_{a, b, c}\overline{B}_{a, b, c}^{-1} - \overline{A}_{a, b, c}^{-1}\overline{B}_{a, b, c} =  - A_{a, b, c}^{-1}B_{a, b, c} - A_{a, b, c}B_{a, b, c}^{-1} = \delta$. 
Thus, the pair $(\overline{A}, \overline{B})$ satisfies the condition~\ref{1} of \ref{def:tribracketbracket}, and we put $\overline{\delta} = \delta$. 
Furthermore, since the equality $- \overline{A}_{a, b, b}^2\overline{B}_{a, b, b}^{-1} = - A_{a, b, b}^{-2}B_{a, b, b} = w^{-1}$ holds for any $a, b \in X_2$, 
the pair $(\overline{A}, \overline{B})$ satisfies the condition~\ref{2} of \ref{def:tribracketbracket}, and we put $\overline{w} = w^{-1}$. 
It is straightforward to see that $(\overline{A}, \overline{B})$ satisfies the equalities~\ref{3a}--\ref{3c}, however, 
At present we do not know whether it satisfies the equalities~\ref{3d} and \ref{3e} in general. 
Namely, we do not know whether $(\overline{A}, \overline{B})$ is a tribracket bracket for a given tribracket bracket $(\overline{A}, \overline{B})$. 

Nevertheless, it is still possible to define an element $\beta_{X_2}^{(\overline{A}, \overline{B})}(D_L, C) \in R$ for 
each $X_2$-coloring $C$ of $D_L$ as follows: 
\[
    \beta_{X_2}^{(\overline{A}, \overline{B})}(D_L, C) = \overline{w}^{n-p} \sum_{{\overline{s}} \in \mathcal{S}(D_L)} \overline{\delta}^{k_{\overline{s}}} \prod_{c \in \mathcal{C}(D_L)} \overline{s}(c)(cl_C(c))^{\text{sign}(c)}, 
\]
where recall that $p$ and $n$ denote the numbers of positive and negative crossings of $D_L$, respectively, 
and for each state ${\overline{s}} : \mathcal{C}(D_L) \to \{\overline{A}, \overline{B}\}$, 
$k_{\overline{s}}$ denotes the number of circles in the diagram 
obtained by smoothing each crossing as in Table~\ref{tab:smoothing}. 

\begin{remark}
  Since $(\overline{A}, \overline{B})$ satisfies the conditions~\ref{1} and \ref{2} of Definition~\ref{def:tribracketbracket}, 
  we see that $\beta_{X_2}^{(\overline{A}, \overline{B})}(D_L, C)$ is invariant under the Reidemeister moves I and II. 
  However, we do not know whether $\beta_{X_2}^{(\overline{A}, \overline{B})}(D_L, C)$ is invariant under the Reidemeister move III since 
  we do not know whether $(\overline{A}, \overline{B})$ satisfies the equalities~\ref{3d} and \ref{3e} in Definition~\ref{def:tribracketbracket}. 
  If $(A, B)$ satisfies $A_{0,0,1}B_{0,1,0} = B_{0,0,1}A_{0,1,0}$ and $A_{1,1,0}B_{1,0,1} = B_{1,1,0}A_{1,0,1}$, then by 
  Proposition~\ref{prop:mirror} given below, we see that $\beta_{X_2}^{(\overline{A}, \overline{B})}(D_L, C)$ is actually invariant under 
  the Reidemeister move III as well. 
\end{remark}

\begin{proposition}
\label{prop:mirror}
Let $R$ be a commutative ring, and $(A, B)$ a tribracket bracket with respect to 
$X_2$ and $R$ with $A_{0,0,1}B_{0,1,0} = B_{0,0,1}A_{0,1,0}$ and $A_{1,1,0}B_{1,0,1} = B_{1,1,0}A_{1,0,1}$. 
Let $D_L$ be a diagram of a link $L$. 
Then, for any $X_2$-coloring $C$ of $D_L$, we have 
$\beta^{(\overline{A}, \overline{B})}_{X_2}(D_L, C) = \beta^{(A, B)}_{X_2}(D^*_L, C)$. 
\end{proposition}
\begin{proof}
Set  $D_L = D_1 \cup \cdots \cup D_n$. 
By Proposition~\ref{prop:coloring_number},  
there exists a decomposition $(\mathcal{D}_1, \mathcal{D}_2)$ of $\{D_1, \ldots, D_n\}$ and an element $a \in X_2$ with 
$C = C(a, \mathcal{D}_1, \mathcal{D}_2)$. 
Let $\mathcal{C}(D_L) = \{c_1, \ldots, c_m\}$ and $\mathcal{C}(D^*_L) = \{c^*_1, \ldots, c^*_m\}$, where 
$c^*_i$ is obtained by changing the crossing $c_i$. 
Note that if $p$ and $n$ are the numbers of positive and negative crossings of $D_L$, respectively, then 
the number of positive and negative crossings of $D^*_L$ are $n$ and $p$, respectively. 
Thus, it suffices to show that 
    \[
    \overline{w}^{n-p} \sum_{{\overline{s}} \in \mathcal{S}(D_L)} \overline{\delta}^{k_{\overline{s}}} \prod_{i = 1}^{m} {\overline{s}}(c_i)(\cl_C(c_i))^{\sign(c_i)}
    = w^{p-n} \sum_{s^* \in \mathcal{S}(D^*_L)} \delta^{k_{s^*}} \prod_{i = 1}^{m} s^*(c^*_i)(\cl_C(c^*_i))^{\sign(c^*_i)}. 
    \]
    
    First, we have $\overline{w}^{n-p} = (w^{-1})^{n-p} = w^{p-n}$ by the definition of $\overline{w}$. 
    For a state ${\overline{s}} : \mathcal{C}(D_L) \to \{{\overline{A}}, {\overline{B}}\}$ of $D_L$, define a state 
 $s^* : \mathcal{C}(D^*_L) \to \{A, B\}$ of $D^*_L$ by 
\[ s^*(c^*_i) = \left\{ 
\begin{array}{ll}
A & (\mbox{if $\overline{s}(c_i) = \overline{A}$})\\
B & (\mbox{if $\overline{s}(c_i) = \overline{B}$}).
\end{array} \right.  \]
Then, we have $k_{\overline{s}} = k_{s^*}$ for any state $\overline{s}$ of $D_L$. 
    
    Choose an arbitrary state $\overline{s}$ of $D_L$. 
We show that 
\[ \prod_{i = 1}^{m} \overline{s}(c_i)(\cl_C(c_i))^{\sign(c_i)} = \prod_{i = 1}^{m} s^*(c^*_i)(\cl_C(c^*_i))^{\sign(c^*_i)}.\] 
Let $c_i$ be a crossing of $D_{i_1}$ and $D_{i_2}$. 
Let $m'$ denote the number of crossings in $|D_L|$ around which the regions are colored as in of the models in Figure~\ref{fig:coloring_pair}. 
Recall that $m'$ is an even number. 
As in the proof of Proposition~\ref{prop:comp_color}, swapping the subscripts if necessary, we can assume that  crossings of $D_L$ satisfy the conditions 
(\ref{prop:comp_color}-a) and (\ref{prop:comp_color}-b).  

Suppose that $i$ is greater than $m'$. 
Then, the regions around $c_i$ are colored as in either of the models shown Figures~\ref{fig:crossing_case1} and 
\ref{fig:crossing_case2} up to over/under information.
Thus, we have $\overline{s}(c_i)(\cl_C(c_i))^{\sign(c_i)} = s^*(c^*_i)(\cl_C(c^*_i))^{\sign(c^*_i)}$ because 
$\cl_C(c_i) = \cl_C(c^*_i)$, $\sign(c_i) = -\sign(c^*_i)$ and $\overline{s}(c_i)(a,b,c) = s^*(c^*_i)(a,b,c)^{-1}$ for 
any $a, b,c \in X_2$. 

Next, suppose that $i$ is an odd number less than $m'$. 
We show that 
\[ 
\overline{s}(c_i)(cl_C(c_i))^{\sign(c_i)}\overline{s}(c_{i+1})(cl_C(c_{i+1}))^{\sign(c_{i+1})} = s^*(c^*_i)(cl_C(c^*_i))^{\sign(c^*_i)}s^*(c^*_{i+1})(cl_C(c^*_{i+1}))^{\sign(c^*_{i+1})}.\] 
We assume that the colors of regions around $c_i, c_{i+1}$ are as on the top of Figure~\ref{fig:ci_ci1}. 
The same proof works for the case of the bottom of Figure~\ref{fig:ci_ci1}. 
    \begin{itemize}
        \item Suppose that $\sign(c_i) = \sign(c_{i+1}) = 1$. 
          In this case, we have  
          \[\cl_C(c_i) = (0,0,1), \cl_C(c_{i+1}) = (0,1,0), \cl_C(c^*_i) = (0,1,0), \cl_C(c^*_{i+1}) = (0,0,1),\] 
          as well as 
          $\sign(c^*_i) = \sign(c^*_{i+1}) = -1$. 
          Further, for any state state $s^*: \mathcal{C}(D^*_L) \to \{A, B\}$, we have 
          \[s^*(c_i)(0,0,1)s^*(c_{i+1})(0,1,0) = s^*(c_{i+1})(0,0,1)s^*(c_i)(0,1,0)\] 
          by assumption. 
          Therefore, we have 
            \begin{align*}
                \overline{s}(c_i)(\cl_C(c_i))^{\sign(c_i)}&\overline{s}(c_{i+1})(\cl_C(c_{i+1}))^{\sign(c_{i+1})}\\
                &= \overline{s}(c_i)(0,0,1)\overline{s}(c_{i+1})(0,1,0)\\
                &= s^*(c_i)(0,0,1)^{-1}s^*(c_{i+1})(0,1,0)^{-1}\\
                &= s^*(c_{i+1})(0,0,1)^{-1}s^*(c_i)(0,1,0)^{-1}\\
                &= s^*(c_i)(0,1,0)^{-1}s^*(c_{i+1})(0,0,1)^{-1}\\
                &= s^*(c^*_i)(\cl_C(c^*_i))^{\sign(c^*_i)}s^*(c^*_{i+1})(\cl_C(c^*_{i+1}))^{\sign(c^*_{i+1})},  
            \end{align*}
            which is the assertion. 
        \item Suppose that $\sign(c_i) = 1$ and $\sign(c_{i+1}) = -1$. 
        \begin{itemize}
            \item If $\overline{s}(c_i) = \overline{s}(c_{i+1})$, we have
                \begin{align*}
                    \overline{s}(c_i)(\cl_C(c_i))^{\sign(c_i)}&\overline{s}(c_{i+1})(\cl_C(c_{i+1}))^{\sign(c_{i+1})}\\
                    &= \overline{s}(c_i)(0,0,1)\overline{s}(c_{i+1})(0,0,1)^{-1}\\
                    &= 1\\
                    &= s^*(c_i)(0,1,0)^{-1}s^*(c_{i+1})(0,1,0)^{-1}\\
                    &= s^*(c^*_i)(\cl_C(c^*_i))^{\sign(c^*_i)}s^*(c^*_{i+1})(\cl_C(c^*_{i+1}))^{\sign(c^*_{i+1})} , 
                \end{align*}
                which implies the assertion. 
                
            \item If $\overline{s}(c_i) = \overline{A}, \overline{s}(c_{i+1}) = \overline{B}$, we have 
            $A_{0,0,1}^{-1}B_{,0,0,1} = A_{0,1,0}^{-1}B_{0,1,0}$ by assumption. 
            It follows that 
                \begin{align*}
                    \overline{s}(c_i)(\cl_C(c_i))^{\sign(c_i)}&\overline{s}(c_{i+1})(\cl_C(c_{i+1}))^{\sign(c_{i+1})}\\
                    &= \overline{s}(c_i)(0,0,1)\overline{s}(c_{i+1})(0,0,1)^{-1}\\
                    &= s^*(c_i)(0,0,1)^{-1}s^*(c_{i+1})(0,0,1)\\
                    &= A_{0,0,1}^{-1}B_{0,0,1}\\
                    &= A_{0,1,0}^{-1}B_{0,1,0}\\
                    &= s^*(c_i)(0,1,0)^{-1}s^*(c_{i+1})(0,1,0)\\
                    &= s^*(c^*_i)(\cl_C(c^*_i))^{\sign(c^*_i)}s^*(c^*_{i+1})(\cl_C(c^*_{i+1}))^{\sign(c^*_{i+1})} , 
                \end{align*}
               which is the assertion. 
            \item 
            The arguments for the case where $\overline{s}(c_i) = \overline{B}$ and $\overline{s}(c_{i+1}) = \overline{A}$ 
            are the same as for the case where $\overline{s}(c_i) = \overline{A}, \overline{s}(c_{i+1}) = \overline{B}$. 
        \end{itemize}
        \item 
        Suppose that $\sign(c_i) = -1$ and $\sign(c_{i+1}) = 1$. 
        In this case, we can show the assertion as in the case where $\sign(c_i) = 1$ and $\sign(c_{i+1}) = -1$.
        \item 
        Suppose that $\sign(c_i) = \sign(c_{i+1}) = -1$
        In this case, we can show the assertion as in the case where $\sign(c_i) = \sign(c_{i+1}) = 1$.
    \end{itemize}
  From above, for any odd number $i\ (i < m')$, we have 
  \[\overline{s}(c_i)(\cl_C(c_i))^{\sign(c_i)}\overline{s}(c_{i+1})(\cl_C(c_{i+1}))^{\sign(c_{i+1})} = s^*(c^*_i)(\cl_C(c^*_i))^{\sign(c^*_i)}s^*(c^*_{i+1})(\cl_C(c^*_{i+1}))^{\sign(c^*_{i+1})} .\]
 Therefore, we have $\prod_{i = 1}^{m} \overline{s}(c_i)(\cl_C(c_i))^{\sign(c_i)} = \prod_{i = 1}^{m} s^*(c^*_i)(\cl_C(c^*_i))^{\sign(c^*_i)}$ 
 for any state $\overline{s}$ of $D_L$. 
 
 Consequently, we obtain 
    \[
    \overline{w}^{n-p} \sum_{{\overline{s}} \in \mathcal{S}(D_L)} \overline{\delta}^{k_{\overline{s}}} \prod_{i = 1}^{m} {\overline{s}}(c_i)(cl_C(c_i))^{\sign(c_i)}
    = w^{p-n} \sum_{s^* \in \mathcal{S}(D^*_L)} \delta^{k_{s^*}} \prod_{i = 1}^{m} s^*(c^*_i)(cl_C(c^*_i))^{\sign(c^*_i)} . 
    \]
\end{proof}

\section{The universal tribracket brackets for the canonical two-element tribraket}
\label{sec: The universal tribracket brackets for the canonical two-element tribraket}

In this section, we introduce five tribracket brackets, $(A^{(1)}, B^{(1)}), \ldots, (A^{(5)}, B^{(5)})$, with respect to 
the canonical two-element tribracket $X_2$ and the Laurent polynomial ring $\mathbb{Z}[x_1^{\pm1}, \ldots, x_5^{\pm1}]$, referred to as the \emph{universal tribracket brackets}. 
We then show that any tribracket bracket with respect to $X_2$ over an arbitrary integral domain $R$ can be obtained from 
one of the pairs $(A^{(i)}, B^{(i)})$ through a ring homomorphism $f : \Zpoly \to R$.

Let $R$ be a commutative ring, and let $A, B: X_2^3 \to R$ be maps. 
We express these maps in terms of $3$-tensors as follows:

\begin{align*}
    A &= \left(
        \begin{pmatrix}
            A_{0,0,0} & A_{0,0,1}\\
            A_{0,1,0} & A_{0,1,1}
        \end{pmatrix},
        \begin{pmatrix}
            A_{1,0,0} & A_{1,0,1}\\
            A_{1,1,0} & A_{1,1,1}
        \end{pmatrix} \right)\\
    B &= \left(
        \begin{pmatrix}
            B_{0,0,0} & B_{0,0,1}\\
            B_{0,1,0} & B_{0,1,1}
        \end{pmatrix},
        \begin{pmatrix}
            B_{1,0,0} & B_{1,0,1}\\
            B_{1,1,0} & B_{1,1,1}
        \end{pmatrix} \right)\\
\end{align*}

\begin{definition}
For each $i \in \{1,2,3,4,5\}$, we define maps $A^{(i)}, B^{(i)} : X_2^3 \to \Zpoly$ as follows: 
    \begin{align*}
        {A^{(1)}} &= \left(
        \begin{pmatrix}
            x_1 & x_2\\
            x_3 & x_1
        \end{pmatrix},
        \begin{pmatrix}
            x_1 & x_4\\
            x_2x_3x_4^{-1} & x_1
        \end{pmatrix} \right) , \\
        {B^{(1)}} &= \left(
        \begin{pmatrix}
            x_5 & x_1^{-1}x_2x_5\\
            x_1x_3x_5^{-1} & x_5
        \end{pmatrix},
        \begin{pmatrix}
            x_5 & x_1x_4x_5^{-1}\\
            x_1^{-1}x_2x_3x_4^{-1}x_5 & x_5
        \end{pmatrix} \right), \\
        {A^{(2)}} &= \left(
        \begin{pmatrix}
            x_1 & x_2\\
            x_3 & x_1^3x_5^{-2}
        \end{pmatrix},
        \begin{pmatrix}
            x_1^3x_5^{-2} & x_4\\
            x_2x_3x_4^{-1} & x_1
        \end{pmatrix} \right), \\
        {B^{(2)}} &= \left(
        \begin{pmatrix}
            x_5 & x_1^{-1}x_2x_5\\
            x_1x_3x_5^{-1} & x_1^4x_5^{-3}
        \end{pmatrix},
        \begin{pmatrix}
            x_1^4x_5^{-3} & x_1x_4x_5^{-1}\\
            x_1^{-1}x_2x_3x_4^{-1}x_5 & x_5
        \end{pmatrix} \right), \\
        {A^{(3)}} &= \left(
        \begin{pmatrix}
            x_1 & x_2\\
            x_3 & x_1
        \end{pmatrix},
        \begin{pmatrix}
            x_1 & x_4\\
            x_2x_3x_4^{-1} & x_1
        \end{pmatrix} \right), \\
        {B^{(3)}} &= \left(
        \begin{pmatrix}
            x_5 & x_1x_2x_5^{-1}\\
            x_1^{-1}x_3x_5 & x_5
        \end{pmatrix},
        \begin{pmatrix}
            x_5 & x_1^{-1}x_4x_5\\
            x_1x_2x_3x_4^{-1}x_5^{-1} & x_5
        \end{pmatrix} \right), \\
        {A^{(4)}} &= \left(
        \begin{pmatrix}
            x_1 & x_2\\
            x_3 & x_1^3x_5^{-2}
        \end{pmatrix},
        \begin{pmatrix}
            x_1^3x_5^{-2} & x_4\\
            x_2x_3x_4^{-1} & x_1
        \end{pmatrix} \right), \\
        {B^{(4)}} &= \left(
        \begin{pmatrix}
            x_5 & x_1x_2x_5^{-1}\\
            x_1^{-1}x_3x_5 & x_1^4x_5^{-3}
        \end{pmatrix},
        \begin{pmatrix}
            x_1^4x_5^{-3} & x_1^{-1}x_4x_5\\
            x_1x_2x_3x_4^{-1}x_5^{-1} & x_5
        \end{pmatrix} \right), \\
        {A^{(5)}} &= \left(
        \begin{pmatrix}
            x_1 & x_2\\
            x_3 & x_1
        \end{pmatrix},
        \begin{pmatrix}
            x_1 & x_4\\
            x_2x_3x_4^{-1} & x_1
        \end{pmatrix} \right), \\
        {B^{(5)}} &= \left(
        \begin{pmatrix}
            x_5 & x_1^{-1}x_2x_5\\
            x_1^{-1}x_3x_5 & x_5
        \end{pmatrix},
        \begin{pmatrix}
            x_5 & x_1^{-1}x_4x_5\\
            x_1^{-1}x_2x_3x_4^{-1}x_5 & x_5
        \end{pmatrix} \right)
    \end{align*}
For each $i \in \{1,2,3,4,5\}$ and $a,b,c \in X_2$, we denote $A^{(i)}(a,b,c)$ and $B^{(i)}(a,b,c)$ by 
$A^{(i)}_{a,b,c}$ and $B^{(i)}_{a,b,c}$, respectively. 
\end{definition}

The following proposition is a straightforward verification that each pair $(A^{(i)}, B^{(i)})$ satisfies the conditions of Definition~\ref{def:tribracketbracket}. 
\begin{proposition}\label{prop:types 1-5 tribracket brackets}
 For each $i \in \{1, 2, 3, 4, 5\}$, 
 $(A^{(i)}, B^{(i)})$ is a tribracket bracket with respect to $X_2$ and $\Zpoly$. 
\end{proposition}

\begin{lemma}\label{lem:domain_tri_eq}
Let $R$ be an integral domain, and 
let $(A,B)$ be a tribracket bracket with respect to $X_2$ and $R$.  
Then the following conditions  hold: 
    \begin{enumerate} \renewcommand{\labelenumi}{(\ref{lem:domain_tri_eq}-\arabic{enumi})}
        \item $B_{0,0,1} = A_{0,0,0}B_{0,0,0}^{-1}A_{0,0,1}$ or $B_{0,0,1} = A_{0,0,0}^{-1}B_{0,0,0}A_{0,0,1}$. \label{dom_tri_eq_1}
        \item $B_{0,1,0} = A_{0,0,0}B_{0,0,0}^{-1}A_{0,1,0}$ or $B_{0,1,0} = A_{0,0,0}^{-1}B_{0,0,0}A_{0,1,0}$. \label{dom_tri_eq_2}
        \item $A_{0,1,1} = A_{0,0,0}$ or $A_{0,1,1} = A_{0,0,0}^3B_{0,0,0}^{-2}$. \label{dom_tri_eq_3}
        \item If $B_{0,1,0} = A_{0,0,0}B_{0,0,0}^{-1}A_{0,1,0}$, then $B_{0,0,1} = A_{0,0,0}^{-1}B_{0,0,0}A_{0,0,1}$. \label{dom_tri_eq_4}
        \item If $B_{0,1,0} = A_{0,0,0}^{-1}B_{0,0,0}A_{0,1,0}$ and $B_{0,0,1} = A_{0,0,0}^{-1}B_{0,0,0}A_{0,0,1}$, then 
        $A_{0,0,0} = A_{0,1,1}$. \label{dom_tri_eq_5}
    \end{enumerate}
\end{lemma}
\begin{proof}
    \begin{enumerate} \renewcommand{\labelenumi}{(\ref{lem:domain_tri_eq}-\arabic{enumi})}
        \item By the condition~(\ref{1}) in Definition~\ref{def:tribracketbracket}, 
        we have 
        \[-A_{0,0,0}B_{0,0,0}^{-1} - A_{0,0,0}^{-1}B_{0,0,0} = -A_{0,0,1}B_{0,0,1}^{-1} - A_{0,0,1}^{-1}B_{0,0,1}.\]  
        Modifying this equation, we obtain 
        $(B_{0,0,1} - A_{0,0,0}B_{0,0,0}B_{0,0,0}^{-1}A_{0,0,1})(B_{0,0,1} - A_{0,0,0}^{-1}B_{0,0,0}A_{0,0,1}) = 0$. 
        Since $R$ is an integral domain, it follows that $B_{0,0,1} = A_{0,0,0}B_{0,0,0}^{-1}A_{0,0,1}$ or 
        $B_{0,0,1} = A_{0,0,0}^{-1}B_{0,0,0}A_{0,0,1}$. 
        \item 
        The arguments are the same as for (\ref{lem:domain_tri_eq}-\ref{dom_tri_eq_1}). 
        \item  
        By the condition~(\ref{1}) of Definition~\ref{def:tribracketbracket}, we have 
        $-A_{0,0,0}B_{0,0,0}^{-1} - A_{0,0,0}^{-1}B_{0,0,0} = -A_{0,1,1}B_{0,1,1}^{-1} - A_{0,1,1}^{-1}B_{0,1,1}$. 
        Modifying this equation after substituting equations~(\ref{lem:tri_eq}-\ref{tri_eq4}) in Lemma~\ref{lem:tri_eq}, we obtain 
        $(1-A_{0,0,0}^{-1}A_{0,1,1})(B_{0,0,0}^2 - A_{0,0,0}^3A_{0,1,1}^{-1}) = 0$. 
        Since $R$ is an integral domain, it follows $A_{0,1,1} = A_{0,0,0}$ or $A_{0,1,1} = A_{0,0,0}^3B_{0,0,0}^{-2}$. 
        \item 
        Suppose that $B_{0,0,1} \neq A_{0,0,0}^{-1}B_{0,0,0}A_{0,0,1}$. 
        By (\ref{lem:domain_tri_eq}-\ref{dom_tri_eq_1}), we have $B_{0,0,1} = A_{0,0,0}B_{0,0,0}^{-1}A_{0,0,1}$. 
        If $A_{0,0,0}B_{0,0,0}^{-1} = A_{0,0,0}^{-1}B_{0,0,0}$, then 
        $B_{0,0,1} = A_{0,0,0}B_{0,0,0}^{-1}A_{0,0,1} = A_{0,0,0}^{-1}B_{0,0,0}A_{0,0,1}$, which contradicts the assumption. 
        Thus, we have $A_{0,0,0}B_{0,0,0}^{-1} \neq A_{0,0,0}^{-1}B_{0,0,0}$. 
        On the other hand, substituting $(a, b, c, d) = (0,0,0,1)$ into equation~(\ref{3d}) in Definition~\ref{def:tribracketbracket}, 
        and modifying that using $B_{1,1,1} = B_{0,0,0}$ and $B_{0,0,1} = A_{0,0,0}B_{0,0,0}^{-1}A_{0,0,1}$, we get 
        $(A_{0,1,0}B_{0,0,0} - B_{0,1,0}A_{0,0,0})(A_{0,0,0}^2B_{0,0,0}^{-2}-1)=0$. 
        Since $A_{0,0,0}B_{0,0,0}^{-1} \neq A_{0,0,0}^{-1}B_{0,0,0}$, we have $A_{0,0,0}^2 \neq B_{0,0,0}^2$, and hence, 
        $A_{0,0,0}^2B_{0,0,0}^{-2}-1 \neq 0$. 
        Thus, we obtain $A_{0,1,0}B_{0,0,0} - B_{0,1,0}A_{0,0,0} = 0$, that is, $B_{0,1,0} = A_{0,0,0}^{-1}B_{0,0,0}A_{0,1,0}$. 
        Now, we have $B_{0,1,0} \neq A_{0,0,0}B_{0,0,0}^{-1}A_{0,1,0}$, because $A_{0,0,0} B_{0,0,0}^{-1} \neq A_{0,0,0}^{-1} B_{0,0,0}$. 
        \item 
        Suppose that $B_{0,1,0} = A_{0,0,0}^{-1}B_{0,0,0}A_{0,1,0}$ and $B_{0,0,1} = A_{0,0,0}^{-1}B_{0,0,0}A_{0,0,1}$. 
         Substituting $(a, b, c, d) = (0,1,0,1)$ in to equation~(\ref{3d}) in Definition~\ref{def:tribracketbracket},  and then modifying that 
         using equations~(\ref{lem:tri_eq}-\ref{tri_eq4}), (\ref{lem:tri_eq}-\ref{tri_eq5}), (\ref{lem:tri_eq}-\ref{tri_eq6}), (\ref{lem:tri_eq}-\ref{tri_eq7}) in 
         Lemma~\ref{lem:tri_eq} and the assumptions 
         $B_{0,1,0} = A_{0,0,0}^{-1}B_{0,0,0}A_{0,1,0}$, $B_{0,0,1} = A_{0,0,0}^{-1}B_{0,0,0}A_{0,0,1}$, 
         we obtain $(A_{0,1,1}A_{0,0,0}^{-1}-1)(1-A_{0,0,0}^{-2}B_{0,0,0}^2) = 0$. 
         If $A_{0,1,1}A_{0,0,0}^{-1}-1 = 0$, then we have $A_{0,0,0} = A_{0,1,1}$, whence the assertion. 
         Suppose $1-A_{0,0,0}^{-2}B_{0,0,0}^2 = 0$. 
         Then we have $A_{0,0,0}^2 = B_{0,0,0}^2$. 
         It follows from the condition~(\ref{1}) in Definition~\ref{def:tribracketbracket}, that 
         $-A_{0,0,0}B_{0,0,0}^{-1} - A_{0,0,0}^{-1}B_{0,0,0} = -A_{0,1,1}B_{0,1,1}^{-1} - A_{0,1,1}^{-1}B_{0,1,1}$. 
         Modifying this equation using $A_{0,0,0}^2 = B_{0,0,0}^2$ and the equation~(\ref{lem:tri_eq}-\ref{tri_eq4}) in Lemma~\ref{lem:tri_eq}, 
         we finally get $A_{0,0,0} = A_{0,1,1}$. 
    \end{enumerate}
\end{proof}

\begin{lemma} \label{lem:type}
Let $R$ be an integral domain, and 
let $(A,B)$ be a tribracket bracket with respect to $X_2$ and $R$.  
Then at least one of the following conditions $(1)$--$(5)$ holds: 
    \begin{enumerate}
        \item $B_{0,0,1} = A_{0,0,0}B_{0,0,0}^{-1}A_{0,0,1}$, $B_{0,0,1} = A_{0,0,0}^{-1}B_{0,0,0}A_{0,0,1}$ and $A_{0,1,1} = A_{0,0,0}$. 
        \item $B_{0,0,1} = A_{0,0,0}B_{0,0,0}^{-1}A_{0,0,1}$, $B_{0,0,1} = A_{0,0,0}^{-1}B_{0,0,0}A_{0,0,1}$ and $A_{0,1,1} = A_{0,0,0}^3B_{0,0,0}^{-2}$. 
        \item $B_{0,0,1} = A_{0,0,0}^{-1}B_{0,0,0}A_{0,0,1}$, $B_{0,1,0} = A_{0,0,0}B_{0,0,0}^{-1}A_{0,1,0}$ and $A_{0,1,1} = A_{0,0,0}$. 
        \item $B_{0,0,1} = A_{0,0,0}^{-1}B_{0,0,0}A_{0,0,1}$, $B_{0,1,0} = A_{0,0,0}B_{0,0,0}^{-1}A_{0,1,0}$ and $A_{0,1,1} = A_{0,0,0}^3B_{0,0,0}^{-2}$. 
        \item $B_{0,0,1} = A_{0,0,0}^{-1}B_{0,0,0}A_{0,0,1}$, $B_{0,1,0} = A_{0,0,0}^{-1}B_{0,0,0}A_{0,1,0}$ and $A_{0,1,1} = A_{0,0,0}$. 
    \end{enumerate}
\end{lemma}
\begin{proof}
   By Lemma~\ref{lem:domain_tri_eq} (\ref{lem:domain_tri_eq}-\ref{dom_tri_eq_1}), (\ref{lem:domain_tri_eq}-\ref{dom_tri_eq_2}), (\ref{lem:domain_tri_eq}-\ref{dom_tri_eq_3}), we have 
    \begin{itemize}
        \item $B_{0,0,1} = A_{0,0,0}B_{0,0,0}^{-1}A_{0,0,1}$ or $B_{0,0,1} = A_{0,0,0}^{-1}B_{0,0,0}A_{0,0,1}$;  
        \item $B_{0,1,0} = A_{0,0,0}B_{0,0,0}^{-1}A_{0,1,0}$ or $B_{0,1,0} = A_{0,0,0}^{-1}B_{0,0,0}A_{0,1,0}$; and 
        \item $A_{0,1,1} = A_{0,0,0}$  or $A_{0,1,1} = A_{0,0,0}^3B_{0,0,0}^{-2}$.
    \end{itemize}
We show that among these eight cases, three are impossible. 
By Lemma~\ref{lem:domain_tri_eq} (\ref{lem:domain_tri_eq}-\ref{dom_tri_eq_4}), 
if $B_{0,0,1} = A_{0,0,0}B_{0,0,0}^{-1}A_{0,0,1}$, then we have $B_{0,0,1} = A_{0,0,0}^{-1}B_{0,0,0}A_{0,0,1}$. 
Thus, there are no tribracket brackets with 
  \begin{itemize}
       \item $B_{0,0,1} = A_{0,0,0}B_{0,0,0}^{-1}A_{0,0,1}$, $B_{0,0,1} = A_{0,0,0}B_{0,0,0}^{-1}A_{0,0,1}$ and $A_{0,1,1} = A_{0,0,0}$; or 
       \item $B_{0,0,1} = A_{0,0,0}B_{0,0,0}^{-1}A_{0,0,1}$, $B_{0,0,1} = A_{0,0,0}B_{0,0,0}^{-1}A_{0,0,1}$ and $A_{0,1,1} = A_{0,0,0}^3B_{0,0,0}^{-2}$. 
  \end{itemize}
On the other hand, by Lemma~\ref{lem:domain_tri_eq} (\ref{lem:domain_tri_eq}-\ref{dom_tri_eq_5}), 
if $B_{0,1,0} = A_{0,0,0}^{-1}B_{0,0,0}A_{0,1,0}$ and $B_{0,0,1} = A_{0,0,0}^{-1}B_{0,0,0}A_{0,0,1}$, then we have $A_{0,0,0} = A_{0,1,1}$. 
Thus, there are no tribracket brackets with 
    \begin{itemize}
        \item $B_{0,1,0} = A_{0,0,0}^{-1}B_{0,0,0}A_{0,1,0}$, $B_{0,0,1} = A_{0,0,0}^{-1}B_{0,0,0}A_{0,0,1}$ and $A_{0,1,1} = A_{0,0,0}^3B_{0,0,0}^{-2}$. 
    \end{itemize}
\end{proof}

By Lemma~\ref{lem:type}, the tribracket brackets with respect to $X_2$ and an integral domain $R$ are divided into five types. 
We say that $(A,B)$ is of \emph{type $i$} if $(A, B)$ satisfies condition $(i)$ of Lemma~\ref{lem:type}. 

\begin{theorem}\label{thm:iff_cond}
Let $R$ be an integral domain. 
For maps $A, B : X_2^3 \to R$, define the ring homomorphism $f_{A,B} : \mathbb{Z}[x_1^{\pm1}, \ldots, x_5^{\pm1}] \to R$ by 
\begin{align*}
    f_{A,B}(1) &= 1_R, & f_{A,B}(x_1) &= A_{0,0,0}, & f_{A,B}(x_2) &= A_{0,0,1}, \\
    f_{A,B}(x_3) &= A_{0,1,0}, & f_{A,B}(x_4) &= A_{1,0,1}, & f_{A,B}(x_5) &= B_{0,0,0}, 
\end{align*}
where $1_R$ is the multiplicative identity of $R$. 
Then $(A, B)$ is a tribracket bracket with respect to $X_2$ and $R$ if and only if 
there exists $i \in \{1, 2, 3, 4, 5\}$ such that $(f_{A,B} \circ A^{(i)}, f_{A,B} \circ B^{(i)}) = (A, B)$. 
\end{theorem}

\begin{proof}
Suppose that there exists $i \in \{1, 2, 3, 4, 5\}$ such that $(f_{A,B} \circ A^{(i)}, f_{A,B} \circ B^{(i)}) = (A, B)$. 
Since $(A^{(i)}, B^{(i)})$ is a tribracket bracket by Proposition~\ref{prop:types 1-5 tribracket brackets} and $f_{A,B}$ is a ring homomorphism, 
$(f_{A,B} \circ A^{(i)}, f_{A,B} \circ B^{(i)})$ is a tribracket bracket. 

Conversely, suppose that $(A, B)$ is a tribracket bracket of type $i$. 
Then, by Lemmas~\ref{lem:tri_eq} and \ref{lem:type}, we can check that 
$f_{A,B}(A^{(i)}_{a,b,c}) = A_{a,b,c}$ and $f_{A,B}(B^{(i)}_{a,b,c}) = B_{a,b,c}$ hold for any $a, b, c \in X_2$. 
Thus, we have $(f_{A,B} \circ A^{(i)}, f_{A,B} \circ B^{(i)}) = (A, B)$. 
\end{proof}

By Theorem~\ref{thm:iff_cond}, any tribracket bracket with respect to $X_2$ and an integral domain $R$ 
is obtained from a certain pair $(A^{(i)}, B^{(i)})$ and a ring homomorphism $f : \mathbb{Z}[x_1^{\pm1}, \ldots, x_5^{\pm1}] \to R$. 
Therefore, we call the pairs $(A^{(1)}, B^{(1)}), \ldots, (A^{(5)}, B^{(5)})$ \emph{universal tribracket brackets} with respect to $X_2$. 

For a tribracket $X$ and a commutative ring $R$, $\mathcal{T}_{X, R}$ denotes the set of tribracket brackets with respect to $X$ and $R$.  
When $X = X_2$ and $R$ is an integral domain, 
$\mathcal{T}_{X_2, R}^{(i)} \subset \mathcal{T}_{X_2, R}$ denotes the set of tribracket brackets of type $i$. 
Then, by Theorem~\ref{thm:iff_cond}, we have $\mathcal{T}_{X_2, R} =  \bigcup_{i=1}^5 \mathcal{T}_{X_2, R}^{(i)}$. 
 $\mathcal{T}_{X_2, R}^{(1)}, \ldots, \mathcal{T}_{X_2, R}^{(5)}$ is not a decomposition of $\mathcal{T}_{X_2, R}$, but their intersections 
 are completely determined by the following Corollary~\ref{cor:intersection}. 

\begin{corollary}\label{cor:intersection}
Let $R$ be an integral domain. 
Then, for any $i, j \in \{1, 2, 3, 4, 5\}$ with $i \neq j$, we have 
\[
\mathcal{T}^{(i)}_{X_2, R} \cap \mathcal{T}^{(j)}_{X_2, R} = \bigcap_{k=1}^5 \mathcal{T}^{(k)}_{X_2, R} = \{(A, B) \in \mathcal{T}_{X_2, R} \mid A_{0,0,0}^2 = B_{0,0,0}^2\}.
\]
\end{corollary}

\begin{proof}
It suffices to show that 
$\mathcal{T}_{X_2, R}^{(i)} \cap \mathcal{T}_{X_2, R}^{(j)} = \{(A, B) \in \mathcal{T}_{X_2, R} \mid A_{0,0,0}^2 = B_{0,0,0}^2\}$ holds for any $i, j \in \{1, 2, 3, 4, 5\}$ with $i \neq j$. 

Suppose that $(A, B) \in \mathcal{T}_{X_2, R}^{(1)} \cap \mathcal{T}_{X_2, R}^{(2)}$. 
Then, since $A_{0,1,1} = A_{0,0,0} = A_{0,0,0}^3B_{0,0,0}^{-2}$ by Lemma~\ref{lem:type}, we have $A_{0,0,0}^2 = B_{0,0,0}^2$. 
Thus $(A, B)$ is contained in $\{(A, B) \in \mathcal{T}_{X_2, R} \mid A_{0,0,0}^2 = B_{0,0,0}^2\}$. 
Conversely, suppose that $(A, B) \in \{(A, B) \in \mathcal{T}_{X_2, R} \mid A_{0,0,0}^2 = B_{0,0,0}^2\}$. 
Then, since $A_{0,0,0}^2 = B_{0,0,0}^2$, we have 
    \begin{itemize}
        \item $B_{0,0,1} = A_{0,0,0}B_{0,0,0}^{-1}A_{0,0,1} = A_{0,0,0}^{-1}B_{0,0,0}A_{0,0,1}$; 
        \item $B_{0,1,0} = A_{0,0,0}B_{0,0,0}^{-1}A_{0,1,0} = A_{0,0,0}^{-1}B_{0,0,0}A_{0,1,0}$; and 
        \item $A_{0,1,1} = A_{0,0,0} = A_{0,0,0}^3B_{0,0,0}^{-2}$.
    \end{itemize}
Thus, by the definitions of tribracket brackets of types 1 and 2, $(A, B)$ is contained in $\mathcal{T}_{X_2, R}^{(1)} \cap \mathcal{T}_{X_2, R}^{(2)}$. 
Consequently, we have $\mathcal{T}_{X_2, R}^{(1)} \cap \mathcal{T}_{X_2, R}^{(2)} = \{(A, B) \in \mathcal{T}_{X_2, R} \mid A_{0,0,0}^2 = B_{0,0,0}^2\}$. 

The other cases can be shown in the same way. 
\end{proof}

By Corollary~\ref{cor:intersection}, we can compute the cardinality of $\mathcal{T}_{X_2, R}$ when $R$ is a finite field as follows. 

\begin{corollary}
We have $|\mathcal{T}_{X_2, \mathbb{Z}/2\mathbb{Z}}| = 1$, and for a prime number $p \neq 2$, we have 
    $|\mathcal{T}_{X_2, \mathbb{Z}/p\mathbb{Z}}| = 5(p-1)^5 - 8(p-1)^4$. 
\end{corollary}

\begin{proof}
 When $p = 2$, we immediately have $|\mathcal{T}_{X_2, \mathbb{Z}/2\mathbb{Z}}| = 1$, as $(\mathbb{Z}/2\mathbb{Z})^\times = \{1\}$. 
 
Suppose that $p$ is a prime number other than $2$. 
Then, we have $|\bigcap_{i=1}^5 \mathcal{T}^{(i)}_{X_2, R}| = |\{(A, B) \in \mathcal{T}_{X_2, R} \mid A_{0,0,0}^2 = B_{0,0,0}^2\}| = 2(p-1)^4$, because 
there are exactly $(p-1)^5$ ring homomorphisms $\mathbb{Z}[x_1^{\pm1}, \ldots, x_5^{\pm1}] \to (\mathbb{Z}/p\mathbb{Z})^\times$. 
Further, there are exactly $2(p-1)^4$ ring homomorphisms $f: \mathbb{Z}[x_1^{\pm1}, \ldots, x_5^{\pm1}] \to (\mathbb{Z}/p\mathbb{Z})^\times$ such that $f(x_1)^2 = f(x_5)^2$. 
This implies that $|\bigcap_{i=1}^5 \mathcal{T}^{(i)}_{X_2, R}| = |\{(A, B) \in \mathcal{T}_{X_2, R} \mid A_{0,0,0}^2 = B_{0,0,0}^2\}| = 2(p-1)^4$. 
Therefore, it follows from Corollary~\ref{cor:intersection} that 
\[
|\mathcal{T}_{X_2, \mathbb{Z}/p\mathbb{Z}}| = \sum_{i=1}^5|\mathcal{T}_{X_2, \mathbb{Z}/p\mathbb{Z}}^{(i)}| - 4|\bigcap_{i=1}^5 \mathcal{T}^{(i)}_{X_2, R}| = 5(p-1)^5 - 8(p-1)^4.
\]
\end{proof}

\section{Quantum enhancement invariants associated with universal tribracket brackets}
\label{sec:Quantum enhancement invariants associated with universal tribracket brackets}


Let $(A^{(i)}, B^{(i)})$ $(i \in \{1, \ldots, 5\})$ be the universal tribracket brackets with respect to $X_2$. In this section, we discuss the relationships between the quantum enhancement polynomials $\Phi^{(A^{(i)}, B^{(i)})}$, which we call the \emph{universal quantum enhancement polynomials}, and certain other link invariants.

\subsection{The universal quantum enhancement polynomials $\Phi^{(A^{(i)}, B^{(i)})}$ and linking numbers}
\label{subsec:The_universal_quantum_enhancement_polynomials_and_linking_numbers}

Let $L$ be a link and $L_1, L_2 \subset L$ be (possibly empty) sublinks of $L$ such that $L = L_1 \cup L_2$ and $L_1 \cap L_2 = \emptyset$. Let $D_L$ be a diagram of $L$, and $D_{L_i} \subset D_L$ be a diagram of $L_i$ for $i=1,2$. We denote by $\mathcal{C}(D_1, D_2)$ the set of crossings where $D_2$ crosses under $D_1$. Then, the \emph{linking number} $\mathrm{lk}(L_1, L_2)$ of $L_1$ and $L_2$ is defined as the sum $\sum_{c \in \mathcal{C}(D_1, D_2)} \text{sign}(c)$.

Next, we use the same notation as in Definition~\ref{def:beta}. For a state $s : \mathcal{C}(D) \to \{A, B\}$ of $D_L$, we define
\[
    \beta^{(A, B)}_{X_2}(D_L, C, s) = w^{n-p} \prod_{c \in \mathcal{C}(D_L)} s(c)(\text{cl}_C(c))^{\text{sign}(c)}.
\]
With this notation, we can express
\[
    \beta_{X_2}^{(A, B)}(D_L, C) = \sum_{s \in \mathcal{S}(D_L)} \delta^{k_s} \beta_{X_2}^{(A, B)}(D_L, C, s).
\]

For a monomial $f = a \prod_{i=1}^5 x_i^{k_i} \in \mathbb{Z}[\text{poly}]$, we set $\deg_{x_i}(f) = k_i$ for each $i \in \{1, \ldots ,5\}$.

\begin{proposition}\label{prop:linking_num}
   Let $L = L_1 \cup L_2$ be a link and $D_L$ a diagram of $L$.  
   Let $\mathcal{D}_1$ and $\mathcal{D}_2$ be the sets of diagrams of components of $L_1$ and $L_2$, respectively.    
   Then, for any $i \in \{1,2,3,4,5\}$, $a \in X_2$, and any state $s : \mathcal{C}(D) \to \{A^{(i)}, B^{(i)}\}$, we have
   \[
        \deg_{x_2}(\beta^{(A^{(i)}, B^{(i)})}_{X_2}(D_L,  C(a, \mathcal{D}_1, \mathcal{D}_2), s))
        = \deg_{x_3}(\beta^{(A^{(i)}, B^{(i)})}_{X_2}(D_L,  C(a, \mathcal{D}_1, \mathcal{D}_2), s)),
   \]
   and both are equal to $\mathrm{lk}(L_1, L_2)$.
\end{proposition}

\begin{proof}
First, we show that for any $i \in \{1,2,3,4,5\}$, $a \in X_2$, and any state $s : \mathcal{C}(D) \to \{A^{(i)}, B^{(i)}\}$, we have $\deg_{x_2}(\beta'^{(A^{(i)}, B^{(i)})}(D_L, C(a, \mathcal{D}_1, \mathcal{D}_2), s)) = \mathrm{lk}(L_1, L_2)$.

Take $i \in \{1,2,3,4,5\}$, $a \in X_2$, a diagram $D_L$, and a state $s : \mathcal{C}(D) \to \{A^{(i)}, B^{(i)}\}$. Let $l_0^+, l_1^+, r_0^-, r_1^-$ denote the number of crossings in $D_L$ where the regions are colored as in Figure~\ref{fig:linking_num}, where the local models of the colorings are ordered from left to right.

\begin{figure}[htbp]
\centering
\includegraphics[width=1.2\textwidth]{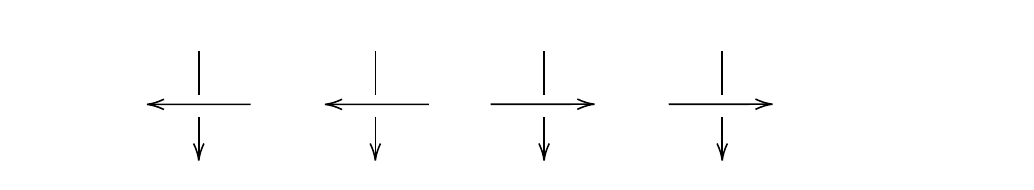}
\begin{picture}(400,0)(0,0)
\put(48,50){$D_i$}
\put(82,85){$D_j$}
\put(73,63){$0$}
\put(97,63){$0$}
\put(73,38){$1$}
\put(97,38){$1$}
\put(135,50){$D_i$}
\put(169,85){$D_j$}
\put(160,63){$1$}
\put(184,63){$1$}
\put(160,38){$0$}
\put(184,38){$0$}
\put(215,50){$D_i$}
\put(249,85){$D_j$}
\put(240,63){$0$}
\put(264,63){$1$}
\put(240,38){$0$}
\put(264,38){$1$}
\put(303,50){$D_i$}
\put(337,85){$D_j$}
\put(328,63){$1$}
\put(352,63){$0$}
\put(328,38){$1$}
\put(352,38){$0$}
\end{picture}
\caption{Local models of colorings by $C(a, \mathcal{D}_1, \mathcal{D}_2)$ at crossings of $\mathcal{C}(\mathcal{D}_1, \mathcal{D}_2)$.}
\label{fig:linking_num}
\end{figure}
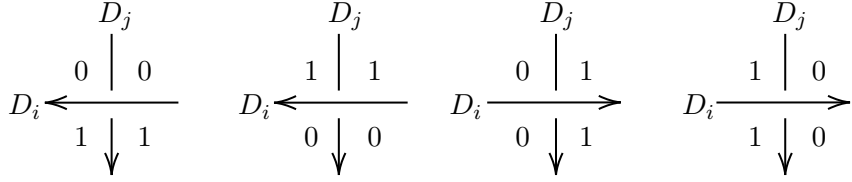

By the definition of $C(a, \mathcal{D}_1, \mathcal{D}_2)$, each crossing in $\mathcal{C}(\mathcal{D}_1, \mathcal{D}_2)$ must follow one of the four models shown in Figure~\ref{fig:linking_num}. Thus, we can express $\mathrm{lk}(L_1, L_2)$ as $l_0^+ + l_1^+ - r_0^- - r_1^-$.

On the other hand, the following assertions hold regardless of $i \in \{1,2,3,4,5\}$ and the state $s : \mathcal{C}(D) \to \{A^{(i)}, B^{(i)}\}$:

\begin{itemize}
    \item If a crossing $c \in \mathcal{C}(D_L)$ is colored as shown on the leftmost or second-from-left in Figure~\ref{fig:coloring_pair}, then $\deg_{x_2}(s(c)(\text{cl}_{C(a, \mathcal{D}_1, \mathcal{D}_2)}(c))^{\text{sign}(c)}) = 1$.
    \item If a crossing $c \in \mathcal{C}(D_L)$ is colored as shown on the rightmost or second-from-right in Figure~\ref{fig:coloring_pair}, then $\deg_{x_2}(s(c)(\text{cl}_{C(a, \mathcal{D}_1, \mathcal{D}_2)}(c))^{\text{sign}(c)}) = -1$.
    \item If a crossing $c \in \mathcal{C}(D_L)$ does not follow any of the configurations shown in Figure~\ref{fig:coloring_pair}, then $\deg_{x_2}(s(c)(\text{cl}_{C(a, \mathcal{D}_1, \mathcal{D}_2)}(c))^{\text{sign}(c)}) = 0$.
\end{itemize}
From this, we have $\deg_{x_2}(\prod_{c \in \mathcal{C}(D_L)} s(c)(cl_C(c))^{\text{sign}(c)}) = l_0^+ + l_1^+ - r_0^- - r_1^-$. 
Furthermore, for any universal tribracket bracket $(A^{(i)}, B^{(i)})$, 
we have $w = -x_1^2x_5^{-1}, \delta = -x_1x_5^{-1} - x_1^{-1}x_5$. 
Since the expressions for $w$ and $\delta$ do not involve $x_2$, i.e., $\deg_{x_2}(w) = \deg_{x_2}(\delta) = 0$, it follows that
\[
    \deg_{x_2}(\beta^{(A^{(i)}, B^{(i)})}(D_L, C(a, \mathcal{D}_1, \mathcal{D}_2), s)) = \mathrm{lk}(L_1, L_2).
\]

Now, to complete the proof, it suffices to show that for any $i \in \{1,2,3,4,5\}$, $a \in X_2$, and any state $s : \mathcal{C}(D) \to \{A^{(i)}, B^{(i)}\}$, we also have
\[
    \deg_{x_2}(\beta^{(A^{(i)}, B^{(i)})}(D_L, C(a, \mathcal{D}_1, \mathcal{D}_2), s)) = \deg_{x_3}(\beta^{(A^{(i)}, B^{(i)})}(D_L, C(a, \mathcal{D}_1, \mathcal{D}_2), s)).
\]
Take $i \in \{1,2,3,4,5\}$, $a \in X_2$, a diagram $D_L$, and a state $s : \mathcal{C}(D) \to \{A^{(i)}, B^{(i)}\}$ arbitrarily. The following assertions hold regardless of $i \in \{1,2,3,4,5\}$ and the state $s : \mathcal{C}(D) \to \{A^{(i)}, B^{(i)}\}$. 
\begin{itemize}
    \item If a crossing $c \in \mathcal{C}(D_L)$ is colored as shown on the rightmost or the second from the left in Figure~\ref{fig:coloring_pair}, then we have $\deg_{x_3}(s(c)(\text{cl}_{C(a, \mathcal{D}_1, \mathcal{D}_2)}(c))^{\text{sign}(c)}) = 1$.
    \item If a crossing $c \in \mathcal{C}(D_L)$ is colored as shown on the leftmost or the second from the right in Figure~\ref{fig:coloring_pair}, then we have $\deg_{x_3}(s(c)(\text{cl}_{C(a, \mathcal{D}_1, \mathcal{D}_2)}(c))^{\text{sign}(c)}) = -1$.
    \item If a crossing $c \in \mathcal{C}(D_L)$ does not match any of the crossings shown in Figure~\ref{fig:coloring_pair}, then we have $\deg_{x_3}(s(c)(\text{cl}_{C(a, \mathcal{D}_1, \mathcal{D}_2)}(c))^{\text{sign}(c)}) = 0$.
\end{itemize}

Since $\deg_{x_3}(w) = \deg_{x_3}(\delta) = 0$ holds for any universal tribracket bracket $(A^{(i)}, B^{(i)})$, we have
\[
    \deg_{x_3}(\beta^{(A^{(i)}, B^{(i)})}(D_L, C(a, \mathcal{D}_1, \mathcal{D}_2), s)) = l_1^+ + r_1^+ - l_0^- - r_0^-.
\]
Furthermore, by Lemma~\ref{lem:coloring_pair}, we have $l_0^+ + l_0^- = r_1^+ + r_1^-$, hence $l_0^+ - r_1^- = r_1^+ - l_0^-$ holds. 
Therefore, we conclude
\begin{align*}
    \deg_{x_3}(\beta^{(A^{(i)}, B^{(i)})}(D_L, C(a, \mathcal{D}_1, \mathcal{D}_2), s))
    &= r_1^+ + l_1^+ - r_0^- - l_0^- \\
    &= l_0^+ + l_1^+ - r_0^- - r_1^- \\
    &= \deg_{x_2}(\beta^{(A^{(i)}, B^{(i)})}(D_L, C(a, \mathcal{D}_1, \mathcal{D}_2), s)).
\end{align*}
\end{proof}

\begin{lemma}\label{lem:non-zero}
Let $D_L$ be a diagram of $L$. 
Then, for any $X_2$-coloring $C$ of $D_L$, we have $\beta^{(A^{(i)}, B^{(i)})}(D_L, C) \neq 0$. 
\end{lemma}

\begin{proof}
Let $f : \mathbb{Z}[\text{poly}] \to \mathbb{Z}$ be the \emph{augmentation map}, that is, 
$f$ is a ring homomorphism such that $f(1) = 1$ and $f(x_j) = 1$ for any $j \in \{1,2,3,4,5\}$. 
Choose $i \in \{1,2,3,4,5\}$ arbitrarily. Then, for each $a, b, c \in X_2$, we have $f(A^{(i)}(a,b,c)) = 1$ and $f(B^{(i)}(a,b,c)) = 1$. 
By the definition of $\beta$, for a diagram $D_{L'}$ obtained by changing a crossing of $D_L$, we get
\[
    f( \beta^{(A^{(i)}, B^{(i)})}(D_L, C)) = f(\beta^{(A^{(i)}, B^{(i)})}(D_{L'}, C))
\]
for any $X_2$-coloring $C$ of $D_L$. Thus, by performing a finite number of crossing changes to transform $D_L$ into a diagram $D_{L''}$ representing the unlink, we have
\[
    f(\beta^{(A^{(i)}, B^{(i)})}(D_L, C)) = f(\beta^{(A^{(i)}, B^{(i)})}(D_{L''}, C)) = (-2)^n \neq 0,
\]
where $n$ is the number of components of $L$. Since $f$ is a ring homomorphism, it follows that $\beta^{(A^{(i)}, B^{(i)})}(D_L, C) \neq 0$.
\end{proof}

\begin{theorem}\label{thm:phi_linking_num}
Let $L = L_1 \cup L_2$ and $L' = L'_1 \cup L'_2$ be links, and let $D_L$ and $D_{L'}$ be diagrams of $L$ and $L'$, respectively.  
Let $\mathcal{D}_1$ and $\mathcal{D}_2$ be the sets of diagrams of the components of $L_1$ and $L_2$, respectively, 
and let $\mathcal{D}'_1$ and $\mathcal{D}'_2$ be the sets of diagrams of the components of $L'_1$ and $L'_2$, respectively. 
Take $i, j \in \{1,2,3,4,5\}$ and $a, a' \in X_2$ arbitrarily, so that at least one of $\beta^{(A^{(i)}, B^{(i)})}(D_L, C(a, \mathcal{D}_1, \mathcal{D}_2))$ 
or $\beta^{(A^{(j)}, B^{(j)})}(D_{L'}, C(a', \mathcal{D}'_1, \mathcal{D}'_2))$ is not zero. If $\mathrm{lk}(L_1, L_2) \neq \mathrm{lk}(L'_1, L'_2)$, 
then we have
\[
    \beta^{(A^{(i)}, B^{(i)})}(D_L, C(a, \mathcal{D}_1, \mathcal{D}_2)) \neq \beta^{(A^{(j)}, B^{(j)})}(D_{L'}, C(a', \mathcal{D}'_1, \mathcal{D}_2)).
\]
\end{theorem}

\begin{proof}
Suppose that $\text{lk}(L_1, L_2) \neq \text{lk}(L'_1, L'_2)$. 
As we have seen in Lemma~\ref{lem:non-zero}, neither $\beta^{(A^{(i)}, B^{(i)})}(D_L, C(a, \mathcal{D}_1, \mathcal{D}_2))$ nor 
$\beta^{(A^{(j)}, B^{(j)})}(D_{L'}, C(a', \mathcal{D}'_1, \mathcal{D}'_2))$ is zero. 
Then, by Proposition~\ref{prop:linking_num}, the degree $\deg_{x_2}$ of each monomial of 
$\beta^{(A^{(i)}, B^{(i)})}(D_L, C(a, \mathcal{D}_1, \mathcal{D}_2))$ coincides with $\mathrm{lk}(L_1, L_2)$, 
while the degree $\deg_{x_2}$ of each monomial of $\beta^{(A^{(j)}, B^{(j)})}(D_{L'}, C(a', \mathcal{D}'_1, \mathcal{D}_2))$ 
coincides with $\mathrm{lk}(L'_1, L'_2)$. Since $\text{lk}(L_1, L_2) \neq \text{lk}(L'_1, L'_2)$ by assumption, it follows that
\[
    \beta^{(A^{(i)}, B^{(i)})}(D_L, C(a, \mathcal{D}_1, \mathcal{D}_2)) \neq \beta^{(A^{(j)}, B^{(j)})}(D_{L'}, C(a', \mathcal{D}'_1, \mathcal{D}_2)).
\]
\end{proof}

For a link $L$, let $\text{LK}(L)$ denote the multiset $\{\text{lk}(L_1, L_2) \mid L = L_1 \cup L_2\}$. From Theorem~\ref{thm:phi_linking_num}, we obtain the following result.

\begin{corollary}\label{cor:linking_num_multi}
Let $L$ and $L'$ be links. If $\Phi^{(A^{(i)}, B^{(i)})}_{X_2} (L) = \Phi^{(A^{(i)}, B^{(i)})}_{X_2} (L')$ for some $i \in \{1,2,3,4,5\}$, then we have $\text{LK}(L) = \text{LK}(L')$.
\end{corollary}

\subsection{The universal quantum enhancement polynomials $\Phi^{(A^{(i)}, B^{(i)})}$ and Jones polynomials for links.}
\label{subsec:The_universal_quantum_enhancement_polynomials_and_Jones_polynomials_for_links}

Let $O$ be the trivial knot. For any link $L$, denote by $K(L)$ the Kauffman bracket polynomial of $L$, normalized as $K(O) = -x^2 - x^{-2}$. 
Also, denote by $J(L)$ the Jones polynomial of $L$, normalized as $J(O) = -t^{\frac{1}{2}} - t^{-\frac{1}{2}}$. 

\begin{theorem}\label{thm:jones_phi5}
Let $L$ and $L'$ be links. If $\Phi^{(A^{(5)}, B^{(5)})}_{X_2}(L) = \Phi^{(A^{(5)}, B^{(5)})}_{X_2}(L')$, then we have $J(L) = J(L')$. 
\end{theorem}

\begin{proof}
Suppose $\Phi^{(A^{(5)}, B^{(5)})}_{X_2}(L) = \Phi^{(A^{(5)}, B^{(5)})}_{X_2}(L')$. 
Define a ring homomorphism $f : \mathbb{Z}[\text{poly}] \to \mathbb{Z}[x]$ by 
\[
    f(1) = 1, \quad f(x_1) = x, \quad f(x_2) = x, \quad f(x_3) = x, \quad f(x_4) = x, \quad f(x_5) = x^{-1}.
\]
Then, for any $a,b,c \in X_2$, we have $f(A^{(5)}(a,b,c)) = x$ and $f(B^{(5)}(a,b,c)) = x^{-1}$. 
Let $D_L$ and $D_L'$ be diagrams of $L$ and $L'$, respectively. 
By the arguments in Example~\ref{ex:jones}, we get $f(\beta^{(A^{(5)}, B^{(5)})}_{X_2}(D_L, C)) = K(L)$ for any $X_2$-coloring $C$ of $D_L$. Therefore, we have 
\[
    \sum_{C \in \text{Col}(D_L)} u^{f(\beta^{(A^{(5)}, B^{(5)})}_{X_2}(D_L, C))} = \sum_{C \in \text{Col}(D_L)} u^{K(L)}
\]
and 
\[
    \sum_{C' \in \text{Col}(D_L')} u^{f(\beta^{(A^{(5)}, B^{(5)})}_{X_2}(D_L', C'))} = \sum_{C' \in \text{Col}(D_L')} u^{K(L')}.
\]
Since we are assuming that $\Phi^{(A^{(5)}, B^{(5)})}_{X_2}(L) = \Phi^{(A^{(5)}, B^{(5)})}_{X_2}(L')$, the equality 
\[
    \sum_{C \in \text{Col}(D_L)} u^{K(L)} = \sum_{C' \in \text{Col}(D_L')} u^{K(L')}
\]
holds. It follows that $K(L) = K(L')$, and equivalently, $J(L) = J(L')$. 
\end{proof}

From Theorem~\ref{thm:jones_phi5}, we see that the universal quantum enhancement polynomial $\Phi^{(A^{(5)}, B^{(5)})}_{X_2}$ is 
at least as strong as the Jones polynomial. In fact, we can assert more.

\begin{proposition}\label{prop_trivial_thistlethwaite}
The universal quantum enhancement polynomial $\Phi^{(A^{(5)}, B^{(5)})}_{X_2}$ is a strictly stronger invariant than the Jones polynomial.
\end{proposition}

\begin{proof}
Let $O$ be the 3-component trivial link, and let $L$ be a link whose diagram is shown in Figure~\ref{fig:thistlethwaite}. 
Then for any $i \in \{1,2,3,4,5\}$, we have $\Phi^{(A^{(i)}, B^{(i)})}_{X_2}(L) \neq \Phi^{(A^{(i)}, B^{(i)})}_{X_2}(O)$. 
\begin{figure}[htpb]
    \centering
    \includegraphics{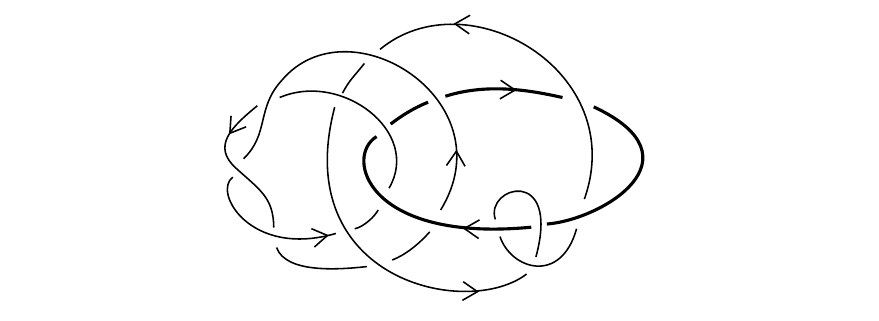}
    \caption{Thistlethwaite's link.}
    \label{fig:thistlethwaite}
\end{figure}

By Thistlethwaite~\cite{Thi01}, we have $J(L) = J(O)$. Now we show that $\Phi^{(A^{(i)}, B^{(i)})}_{X_2}(L) \neq \Phi^{(A^{(i)}, B^{(i)})}_{X_2}(O)$. 
Choose $i \in \{1,2,3,4,5\}$ arbitrarily. For any disjoint sublinks $O_1, O_2$, we have $\text{lk}(O_1, O_2) = 0$. 
Thus, by Proposition~\ref{prop:linking_num}, for any $X_2$-coloring $C$ of any diagram $D_O$ of $O$, the $\deg_{x_2}$ of each monomial of 
$\beta^{(A^{(i)}, B^{(i)})}_{X_2}(D_O, C)$ is zero. 

Let $D_L$ be the link diagram shown in Figure~\ref{fig:thistlethwaite}, where $\mathcal{D}_1$ denotes the diagram drawn with thin lines and 
$\mathcal{D}_2$ denotes the diagram drawn with thick lines. 
Let $L_1$ and $L_2$ be the sublinks of $L$ corresponding to 
$\mathcal{D}_1$ and $\mathcal{D}_2$. We can check that $\text{lk}(L_1, L_2) = -2$. 
Thus, by Proposition~\ref{prop:linking_num}, for any $a \in X_2$, the $\deg_{x_2}$ of each monomial of 
$\beta^{(A^{(i)}, B^{(i)})}_{X_2}(D_L, C(a, \mathcal{D}_1, \mathcal{D}_2))$ is $-2$. 
Therefore, we have $\Phi^{(A^{(i)}, B^{(i)})}_{X_2}(L) \neq \Phi^{(A^{(i)}, B^{(i)})}_{X_2}(O)$.
\end{proof}

The quantum enhancement polynomial of a $3$-component trivial link $O$ 
is calculated as $\Phi_{X_2}^{(A^{(i)}, B^{(i)})}(O) = 16u^{-x_1^{-3}x_5^3 -3x_1^{-1}x_5 -3x_1x_5^{-1} -x_1^{3}x_5^{-3}}$ for any $i \in \{1,2,3,4,5\}$. 

A list of computer-assisted calculations for $\Phi_{X_2}^{(A^{(i)}, B^{(i)})}(L)$ 
, specifically for Thistlethwaite's link $L$, is provided in the Appendix.

\subsection{The quantum enhancement polynomials $\Phi^{(A^{(i)}, B^{(i)})}$ and Jones polynomials for knots.}
\label{subsec:The_quantum_enhancement_polynomials_and_Jones_polynomials_for_knots}

In Proposition~\ref{prop_trivial_thistlethwaite}, we have seen that the invariant $\Phi^{(A^{(5)}, B^{(5)})}_{X_2}$ is strictly stronger than the Jones polynomial for links. On the other hand, when restricting ourselves to knots only, we will see that $\Phi^{(A^{(i)}, B^{(i)})}_{X_2}$ is equivalent to the Jones polynomial for any $i \in \{1,2,3,4,5\}$.

\begin{notation}
Let $D_L$ be a diagram of a link $L$. Let $\text{Col}_{\text{TC}}(D_L)$ denote the set consisting of the two trivial $X_2$-colorings and the two checkerboard $X_2$-colorings of $D_L$. Then, for any $i \in \{1,2,3,4,5\}$, we define
\[
    \Phi^{(A^{(i)}, B^{(i)})}_{\text{TC}}(L) = \sum_{C \in \text{Col}_{\text{TC}}(D_L)} u^{\beta^{(A^{(i)}, B^{(i)})}(D_L, C)}.
\]
\end{notation}

\begin{lemma}\label{lem:sum_zero}
Let $D_L$ be a diagram of a link $L$, $C$ an $X_2$-coloring of $D_L$, and $s$ a state of $D_L$. Then, for any $i \in \{1,2,3,4,5\}$, we have $\deg(\beta'^{(i)}(D_L, C, s)) = 0$, where for a monomial $t \in \mathbb{Z}[\text{poly}]$, $\deg(t)$ denotes the sum $\sum_{i=1}^5 \deg_{x_i}(t)$.
\end{lemma}

\begin{proof}
Let $p$ and $n$ denote the numbers of positive and negative crossings of $D_L$, respectively. Since $\deg(A^{(i)}_{a,b,c}) = \deg(B^{(i)}_{a,b,c}) = 1$ holds for any $i \in \{1,2,3,4,5\}$ and any $a, b, c \in X_2$, we get $\prod_{c \in \mathcal{C}(D_L)} s(c)(\text{cl}_C(c))^{\text{sign}(c)} = p-n$. The degree $\deg(w)$ of the distinguished element $w = -x_1^2x_5^{-1}$ of $(A^{(i)}, B^{(i)})$, which is independent of $i$, is one. Therefore, 
\[
\deg(\beta'^{(A^{(i)}, B^{(i)})}(D_L, C, s)) = \deg(w^{n-p} \prod_{c \in \mathcal{C}(D_L)} s(c)(\text{cl}_C(c))^{\text{sign}(c)}) = (n-p)+(p-n) = 0.
\]
\end{proof}

\begin{lemma}\label{lem:x1x5}
Let $D_L$ be a diagram of a link $L$. Then, for any $C \in \text{Col}_{\text{TC}}(D_L)$ and any $i \in \{1,2,3,4,5\}$, we have $\beta^{(A^{(i)}, B^{(i)})}(D_L, C) \in \mathbb{Z}[(x_1x_5^{-1})^{\pm1}]$.
\end{lemma}

\begin{proof}
Let $C \in \text{Col}_{\text{TC}}(D_L)$ and $i \in \{1,2,3,4,5\}$, $a \in X_2$. We first show that $\beta'^{(i)}(D_L, C, s) \in \mathbb{Z}[(x_1x_5^{-1})^{\pm1}]$. Take a crossing $c$ of $D_L$ and a state $s : \mathcal{C}(D_L) \to \{A^{(i)}, B^{(i)}\}$ arbitrarily. Since $C$ is the trivial or checkerboard coloring, $\text{cl}_C(c)$ is one of $(0,0,0), (0,1,1), (1,0,0), (1,1,1)$. Thus, by the definition of $A^{(i)}, B^{(i)}$, we have $s(c)(\text{cl}_C(c)) \in \mathbb{Z}[x_1^{\pm1}, x_5^{\pm1}]$. 

As we have already seen, the distinguished element $w = -x_1^2x_5^{-1}$ of $(A^{(i)}, B^{(i)})$ is in $\mathbb{Z}[x_1^{\pm1}, x_5^{\pm1}]$. It follows that $\beta'^{(A^{(i)}, B^{(i)})}(D_L, C, s) \in \mathbb{Z}[x_1^{\pm1}, x_5^{\pm1}]$. Furthermore, by Lemma~\ref{lem:sum_zero}, we have $\deg(\beta'^{(A^{(i)}, B^{(i)})}(D_L, C, s)) = 0$, thus, $\beta'^{(A^{(i)}, B^{(i)})}(D_L, C, s) \in \mathbb{Z}[(x_1x_5^{-1})^{\pm1}]$. 

Since the element $\delta$, satisfying the condition~\ref{1} of Definition~\ref{def:tribracketbracket}, can be written as $\delta = -x_1x_5^{-1} - x_1^{-1}x_5$, each term of $\delta$ is in $\mathbb{Z}[(x_1x_5^{-1})^{\pm1}]$. Recalling that $\mathcal{S}(D_L)$ denotes the set of states of $D_L$, and $k_s$ denotes the number of circles in the diagram obtained by smoothing each crossing, we can write
\[
\beta^{(A^{(i)}, B^{(i)})}(D_L, C) = \sum_{s \in \mathcal{S}(D_L)} \delta^{k_s} \beta'^{(A^{(i)}, B^{(i)})}(D_L, C, s),
\]
thus, $\beta^{(A^{(i)}, B^{(i)})}(D_L, C)$ is in $\mathbb{Z}[(x_1x_5^{-1})^{\pm1}]$. 
\end{proof}

\begin{lemma}\label{lem:phi_tc}
Let $D_L$ be a diagram of a link $L$, and $C$ a trivial $X_2$-coloring of $D_L$. Let $f : \mathbb{Z}[x_1^{\pm1}, x_5^{\pm1}] \to \mathbb{Z}[x_1^{\pm1}, x_5^{\pm1}]$ be a ring homomorphism defined by $f(1) = 1, f(x_1) = x_1^{-1}, f(x_5) = x_5^{-1}$. Then, we have
\[
\Phi^{(A^{(i)}, B^{(i)})}_{\text{TC}}(L) = \left\{ 
\begin{array}{ll} 
4u^{\beta^{(A^{(i)}, B^{(i)})}(D_L, C)} &  (i \in \{1,3,5\}), \\
2u^{\beta^{(A^{(i)}, B^{(i)})}(D_L, C)} + 2u^{f(\beta^{(A^{(i)}, B^{(i)})}(D_L, C))} & (i \in \{2,4\}).
\end{array}
\right.
\]
\end{lemma}

\begin{proof}
Let $C$ and $C'$ be the two trivial $X_2$-colorings of $D_L$. 
For each crossing $c$ of $D_L$, $\text{cl}_C(c)$ is one of $(0,0,0)$ and $(1,1,1)$. 
Since $A^{(i)}_{0,0,0}$ coincides with $A^{(i)}_{1,1,1}$ and $B^{(i)}_{0,0,0}$ coincides with $B^{(i)}_{1,1,1}$ for any $i \in \{1,2,3,4,5\}$, $\beta^{(A^{(i)}, B^{(i)})}(D_L, C)$ coincides with $\beta^{(A^{(i)}, B^{(i)})}(D_L, C')$ for any $i \in \{1,2,3,4,5\}$. 
Similarly, we can check that the two checkerboard $X_2$-colorings $C_{\text{check}}$ and $C'_{\text{check}}$ of $D_L$ satisfy $\beta^{(A^{(i)}, B^{(i)})}(D_L, C_{\text{check}}) = \beta^{(A^{(i)}, B^{(i)})}(D_L, C'_{\text{check}})$. 

Suppose that $i = 1,3,5$. In this case, we have $A^{(i)}_{0,0,0} = A^{(i)}_{0,1,1} = A^{(i)}_{1,0,0} = A^{(i)}_{1,1,1} = x_1$ and $B^{(i)}_{0,0,0} = B^{(i)}_{0,1,1} = B^{(i)}_{1,0,0} = B^{(i)}_{1,1,1} = x_5$, hence we have $\beta^{(A^{(i)}, B^{(i)})}(D_L, C) = \beta^{(A^{(i)}, B^{(i)})}(D_L, C_{\text{check}})$. 
It thus follows that $\Phi^{(A^{(i)}, B^{(i)})}_{\text{TC}}(L) = 4u^{\beta^{(A^{(i)}, B^{(i)})}(D_L, C)}$. 

Now suppose that $i = 2, 4$. 
Using similar arguments for trivial $X_2$-colorings and checkerboard colorings, we find that $\beta^{(A^{(i)}, B^{(i)})}(D_L, C)$ coincides with $f(\beta^{(A^{(i)}, B^{(i)})}(D_L, C_{\text{check}}))$. Therefore, we conclude
\[
\Phi^{(A^{(i)}, B^{(i)})}_{\text{TC}}(L) = 2u^{\beta^{(A^{(i)}, B^{(i)})}(D_L, C)} + 2u^{f(\beta^{(A^{(i)}, B^{(i)})}(D_L, C))}.
\]
\end{proof}

\begin{theorem}\label{thm:jones_tc}
Let $L$ and $L'$ be links. Fix any $i \in \{1,2,3,4,5\}$. 
Then, $\Phi^{(A^{(i)}, B^{(i)})}_{\text{TC}}(L) = \Phi^{(A^{(i)}, B^{(i)})}_{\text{TC}}(L')$ if and only if $J(L) = J(L')$.
\end{theorem}

\begin{proof}
Let $C$ and $C'$ be trivial $X_2$-colorings of $D_L$ and $D_{L'}$, respectively. 
Then, by Lemma~\ref{lem:x1x5}, we have $\beta^{(A^{(i)}, B^{(i)})}(D_L, C) \in \mathbb{Z}[(x_1x_5^{-1})^{\pm1}]$. Defining a ring homomorphism $f : \mathbb{Z}[(x_1x_5^{-1})^{\pm1}] \to \mathbb{Z}[x^{\pm2}]$ by $f(1) = 1, f(x_1x_5^{-1}) = x^2$, we get $f(\beta^{(A^{(i)}, B^{(i)})}(D_L, C)) = K(L)$ as in the proof of Theorem~\ref{thm:jones_phi5}. 
Noting that $f$ is a bijection, we see that $\beta^{(A^{(i)}, B^{(i)})}(D_L, C) = \beta^{(A^{(i)}, B^{(i)})}(D_{L'}, C')$ if and only if $K(L) = K(L')$, which is equivalent to the equality $J(L) = J(L')$.

Furthermore, as we have seen in Lemma~\ref{lem:phi_tc}, $\Phi^{(A^{(i)}, B^{(i)})}_{\text{TC}}(L)$ is determined by $\beta^{(A^{(i)}, B^{(i)})}(D_L, C)$. 
Therefore, the above observation immediately implies that $\Phi^{(A^{(i)}, B^{(i)})}_{\text{TC}}(L) = \Phi^{(A^{(i)}, B^{(i)})}_{\text{TC}}(L')$ if and only if $J(L) = J(L')$.
\end{proof}

\begin{corollary}
\label{cor: universal tribracket brackets invariants are euivalent to the Jones polynomial for knots}
Let $K$ and $K'$ be knots. Fix any $i \in \{1,2,3,4,5\}$. Then, $\Phi^{(A^{(i)}, B^{(i)})}(K) = \Phi^{(A^{(i)}, B^{(i)})}(K')$ if and only if $J(K) = J(K')$.
\end{corollary}

\begin{proof}
Let $D_K$ be a diagram of $K$. 
By Corollary~\ref{cor:knot_coloring}, $\text{Col}_{\text{TC}}(D_K)$ is nothing but $\text{Col}(D_K)$. 
Thus, we have $\Phi^{(A^{(i)}, B^{(i)})}_{\text{TC}}(K) = \Phi^{(A^{(i)}, B^{(i)})}(K)$, which implies the assertion from Theorem~\ref{thm:jones_tc}.
\end{proof}

\section{Computer Experiments}
\label{sec:Computer Experiments}

From the results in Section~\ref{sec:Quantum enhancement invariants associated with universal tribracket brackets}, 
we observe that the quantum enhancement polynomial based on the universal tribracket brackets for the canonical two-element tribracket $X_2$ is 
closely related to the Jones polynomial and the linking numbers. 

In this chapter, we analyze 1268 multi-component links with up to ten crossings, using the list from LinkInfo~\cite{LM24}, 
and present the results of computer calculations for five quantum enhancement polynomials $\Phi^{(A^{(i)}, B^{(i)})}_{X_2}(L)$, for $i \in \{1,2,3,4,5\}$, alongside the Jones polynomial $J(L)$, and the multiset $\text{LK}(L) = \{\text{lk}(L_1, L_2) \mid L = L_1 \cup L_2\}$ for these links.

Let $\mathcal{L}$ denote the set of 1268 multi-component links with up to ten crossings.

First, we summarize the results of the computations of $\Phi^{(A^{(i)}, B^{(i)})}_{X_2}(L)$ for $i \in \{1,2,3,4,5\}$ and all $L \in \mathcal{L}$. 

\begin{proposition}
\label{prop: computer calculations of Phi's}
The following holds:
    \begin{itemize}
        \item For each $L \in \mathcal{L}$, we have $\Phi^{(A^{(1)}, B^{(1)})}_{X_2}(L) = \Phi^{(A^{(3)}, B^{(3)})}_{X_2}(L)$. 
        \item For each $L \in \mathcal{L}$, we have $\Phi^{(A^{(2)}, B^{(2)})}_{X_2}(L) = \Phi^{(A^{(4)}, B^{(4)})}_{X_2}(L)$. 
        \item Let $i \in \{2, 4\}$ and $j \in \{1, 3\}$. If $\Phi^{(A^{(i)}, B^{(i)})}_{X_2}(L) = \Phi^{(A^{(i)}, B^{(i)})}_{X_2}(L')$ for $L, L' \in \mathcal{L}$, then $\Phi^{(A^{(j)}, B^{(j)})}_{X_2}(L) = \Phi^{(A^{(j)}, B^{(j)})}_{X_2}(L')$. 
        \item Let $i \in \{1, 2, 3, 4\}$. If $\Phi^{(A^{(i)}, B^{(i)})}_{X_2}(L) = \Phi^{(A^{(i)}, B^{(i)})}_{X_2}(L')$ for $L, L' \in \mathcal{L}$, then we have $\Phi^{(A^{(5)}, B^{(5)})}_{X_2}(L) = \Phi^{(A^{(5)}, B^{(5)})}_{X_2}(L')$. 
    \end{itemize}
\end{proposition}

For each $i, j \in \{1,2,3,4,5\}$ $(i \neq j)$, let $X_2^{(i, j)}$ denote the set of all unordered pairs $\{ L, L' \}$ of distinct links $L, L' \in \mathcal{L}$ such that $\Phi^{(A^{(i)}, B^{(i)})}_{X_2}(L) = \Phi^{(A^{(i)}, B^{(i)})}_{X_2}(L')$ and $\Phi^{(A^{(j)}, B^{(j)})}_{X_2}(L) \neq \Phi^{(A^{(j)}, B^{(j)})}_{X_2}(L')$.

\begin{proposition}
\label{prop: computer calculations of the pairs of Phi's}
The following holds:
    \begin{itemize}
        \item For $i \in \{1, 3\}$ and $j \in \{2, 4\}$, we have $|X_2^{(i, j)}| = 193$. 
        \item For $i \in \{1, 3\}$, we have $|X_2^{(i, 5)}| = 22$. 
        \item For $i \in \{2, 4\}$, we have $|X_2^{(i, 5)}| = 215$. 
    \end{itemize}
\end{proposition}

From Propositions~\ref{prop: computer calculations of Phi's} and~\ref{prop: computer calculations of the pairs of Phi's}, we can make the following conjecture:
\begin{conjecture}
    \begin{itemize}
        \item For each link $L$, we have $\Phi^{(A^{(1)}, B^{(1)})}_{X_2}(L) = \Phi^{(A^{(3)}, B^{(3)})}_{X_2}(L)$. 
        \item For each link $L$, we have $\Phi^{(A^{(2)}, B^{(2)})}_{X_2}(L) = \Phi^{(A^{(4)}, B^{(4)})}_{X_2}(L)$. 
        \item Let $i \in \{2, 4\}$ and $j \in \{1, 3\}$. Then, the invariant $\Phi^{(A^{(i)}, B^{(i)})}_{X_2}$ is strictly stronger than $\Phi^{(A^{(j)}, B^{(j)})}_{X_2}$. 
        \item Let $i \in \{1, 2, 3, 4\}$. Then the invariant $\Phi^{(A^{(i)}, B^{(i)})}_{X_2}$ is strictly stronger than $\Phi^{(A^{(5)}, B^{(5)})}_{X_2}$. 
    \end{itemize}
\end{conjecture}

Next, we obtained the following results for quantum enhancement polynomials and Jones polynomials:
\begin{proposition}\label{prop:Jones_exp}
Let $i \in \{1,2,3,4,5\}$ and $L, L' \in \mathcal{L}$. If $\Phi^{(A^{(i)}, B^{(i)})}_{X_2}(L) = \Phi^{(A^{(i)}, B^{(i)})}_{X_2}(L')$, then we have $J(L) = J(L')$. 
\end{proposition}

When $i=5$, this proposition follows from Theorem~\ref{thm:jones_phi5}. The following conjecture can be made based on these computations:
\begin{conjecture}
Let $i \in \{1,2,3,4\}$ and let $L, L'$ be links. If $\Phi^{(A^{(i)}, B^{(i)})}_{X_2}(L) = \Phi^{(A^{(i)}, B^{(i)})}_{X_2}(L')$, then $J(L) = J(L')$. 
\end{conjecture}

Recall that, for a link $L$, $\text{LK}(L)$ denotes the multiset $\{\text{lk}(L_1, L_2) \mid L = L_1 \cup L_2\}$. From Corollary~\ref{cor:linking_num_multi} and Proposition~\ref{prop:Jones_exp}, we observe that for any $i \in \{1,2,3,4,5\}$ and any links $L, L' \in \mathcal{L}$, if $\Phi^{(A^{(i)}, B^{(i)})}_{X_2}(L) = \Phi^{(A^{(i)}, B^{(i)})}_{X_2}(L')$, then $(J(L), \text{LK}(L)) = (J(L'), \text{LK}(L'))$. However, computer calculations indicate that the converse does not hold.

For $i \in \{1,2,3,4,5\}$, let $X_2^{(i)}$ denote the set of all unordered pairs $\{ L, L' \}$ of distinct links $L, L' \in \mathcal{L}$ such that $\Phi^{(A^{(i)}, B^{(i)})}_{X_2}(L) \neq \Phi^{(A^{(i)}, B^{(i)})}_{X_2}(L')$ and $(J(L), \text{LK}(L)) = (J(L'), \text{LK}(L'))$.

\begin{proposition}
\label{prop:computer experiments for Phi and (LK, J)}
The following holds:
\begin{enumerate}
    \item The set $X_2^{(1)}$ coincides with $X_2^{(3)}$, and $|X_2^{(1)}| = |X_2^{(3)}| = 22$. 
    \item The set $X_2^{(2)}$ coincides with $X_2^{(4)}$, and $|X_2^{(2)}| = |X_2^{(4)}| = 215$. 
    \item The set $X_2^{(5)}$ is empty. 
\end{enumerate}
\end{proposition}

Based on this proposition, we can make the following conjecture:
\begin{conjecture}
The two link invariants $\Phi^{(A^{(i)}, B^{(i)})}_{X_2}$ and $(J, \text{LK})$ satisfy the following:
    \begin{itemize}
        \item For $i \in \{1,2,3,4\}$, the invariant $\Phi^{(A^{(i)}, B^{(i)})}_{X_2}$ is strictly stronger than $(J, \text{LK})$. 
        \item The invariant $\Phi^{(A^{(5)}, B^{(5)})}_{X_2}$ is equivalent to $(J, \text{LK})$. 
    \end{itemize}
\end{conjecture}

\begin{example}
The unordered pair $\{L8a20\{0;0\}, L8a20\{1;0\}\}$ of links $L8a20\{0;0\}$ and $L8a20\{1;0\}$ shown in 
Figure~\ref{fig:ex_X_element} belongs to $X_2^{(i)}$ for any $i \in \{1,2,3,4\}$. 
    \begin{figure}[htbp]
        \centering
        \includegraphics{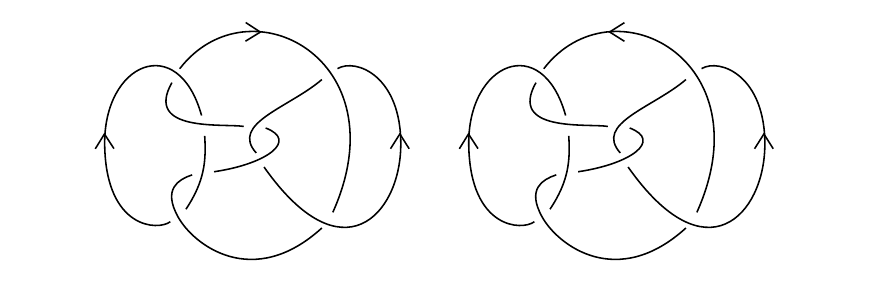}
        \caption{The links $L8a20\{0;0\}$ and $L8a20\{1;0\}$.}
        \label{fig:ex_X_element}
    \end{figure}
   In fact, the invariants $\Phi^{(A^{(i)}, B^{(i)})}_{X_2}$, $J$ and $LK$  for those links are as follows: 
    \begin{itemize}
        \item $J( L8a20\{0;0\} ) = J( L8a20\{1;0\} ) = - t^{-\frac{9}{2}} + t^{-\frac{7}{2}} - 3t^{-\frac{5}{2}} - t^{-\frac{1}{2}} - t^{\frac{1}{2}} -3 t^{\frac{5}{2}} + t^{\frac{7}{2}} - t^{\frac{9}{2}}$
        \item $\text{LK}( L8a20\{0;0\} ) = \text{LK}( L8a20\{1;0\} ) = \{-2, -2, -2, -2, 0, 0, 0, 0, 0, 0, 0, 0, 0, 0, 0, 0, 2, 2, 2, 2\}.$
        \item $\Phi^{(A^{(1)}, B^{(1)})}_{X_2}( L8a20\{0;0\} ) = \Phi^{(A^{(3)}, B^{(3)})}_{X_2}( L8a20\{0;0\} ) = 4u^a + 4u^b + 4u^c + 4u^d$.
        \item $\Phi^{(A^{(2)}, B^{(2)})}_{X_2}( L8a20\{0;0\} ) = \Phi^{(A^{(4)}, B^{(4)})}_{X_2}( L8a20\{0;0\} ) = 4u^a + 4u^b + 2u^c + 2u^d + 2u^e + 2u^f$.
        \item $\Phi^{(A^{(5)}, B^{(5)})}_{X_2}( L8a20\{0;0\} ) = \Phi^{(A^{(5)}, B^{(5)})}_{X_2}( L8a20\{1;0\} ) = 8u^a + 4u^g + 4u^h$.
        \item $\Phi^{(A^{(1)}, B^{(1)})}_{X_2}( L8a20\{1;0\} ) = \Phi^{(A^{(3)}, B^{(3)})}_{X_2}( L8a20\{1;0\} ) = 4u^a + 4u^b + 4u^i + 4u^j$.
        \item $\Phi^{(A^{(2)}, B^{(2)})}_{X_2}( L8a20\{1;0\} ) = \Phi^{(A^{(4)}, B^{(4)})}_{X_2}( L8a20\{1;0\} ) = 4u^a + 4u^b + 2u^i + 2u^j + 2u^k + 2u^l$.
    \end{itemize}
    Here, $a,b,c,d,e,f,g, h, i, j, k, l \in \Zpoly$ are defined as follows: 
    \begin{itemize}
        \item $a = -x_{1}^{-9}x_{5}^{9}+x_{1}^{-7}x_{5}^{7}-3x_{1}^{-5}x_{5}^{5}-x_{1}^{-1}x_{5}-x_{1}x_{5}^{-1}-3x_{1}^{5}x_{5}^{-5}+x_{1}^{7}x_{5}^{-7}-x_{1}^{9}x_{5}^{-9}$
        \item $b = -x_{1}^{-3}x_{5}^{3}-3x_{1}^{-1}x_{5}-3x_{1}x_{5}^{-1}-x_{1}^{3}x_{5}^{-3}$
        \item $c = -x_{1}^{-3}x_{2}^{-2}x_{3}^{-2}x_{5}^{7}-2x_{1}^{-1}x_{2}^{-2}x_{3}^{-2}x_{5}^{5}-2x_{1}x_{2}^{-2}x_{3}^{-2}x_{5}^{3}-x_{1}^{3}x_{2}^{-2}x_{3}^{-2}x_{5}-x_{1}^{7}x_{2}^{-2}x_{3}^{-2}x_{5}^{-3}-x_{1}^{9}x_{2}^{-2}x_{3}^{-2}x_{5}^{-5}$
        \item $d = -x_{1}^{-9}x_{2}^{2}x_{3}^{2}x_{5}^{5}-x_{1}^{-7}x_{2}^{2}x_{3}^{2}x_{5}^{3}-x_{1}^{-3}x_{2}^{2}x_{3}^{2}x_{5}^{-1}-2x_{1}^{-1}x_{2}^{2}x_{3}^{2}x_{5}^{-3}-2x_{1}x_{2}^{2}x_{3}^{2}x_{5}^{-5}-x_{1}^{3}x_{2}^{2}x_{3}^{2}x_{5}^{-7}$
        \item $e = -x_{1}^{7}x_{2}^{-2}x_{3}^{-2}x_{5}^{-3}-x_{1}^{9}x_{2}^{-2}x_{3}^{-2}x_{5}^{-5}-x_{1}^{13}x_{2}^{-2}x_{3}^{-2}x_{5}^{-9}-2x_{1}^{15}x_{2}^{-2}x_{3}^{-2}x_{5}^{-11}-2x_{1}^{17}x_{2}^{-2}x_{3}^{-2}x_{5}^{-13}-x_{1}^{19}x_{2}^{-2}x_{3}^{-2}x_{5}^{-15}$
        \item $f = -x_{1}^{-19}x_{2}^{2}x_{3}^{2}x_{5}^{15}-2x_{1}^{-17}x_{2}^{2}x_{3}^{2}x_{5}^{13}-2x_{1}^{-15}x_{2}^{2}x_{3}^{2}x_{5}^{11}-x_{1}^{-13}x_{2}^{2}x_{3}^{2}x_{5}^{9}-x_{1}^{-9}x_{2}^{2}x_{3}^{2}x_{5}^{5}-x_{1}^{-7}x_{2}^{2}x_{3}^{2}x_{5}^{3}$
        \item $g = -x_{1}^{-5}x_{2}^{-2}x_{3}^{-2}x_{5}^{9}+x_{1}^{-3}x_{2}^{-2}x_{3}^{-2}x_{5}^{7}-3x_{1}^{-1}x_{2}^{-2}x_{3}^{-2}x_{5}^{5}-x_{1}^{3}x_{2}^{-2}x_{3}^{-2}x_{5}-x_{1}^{5}x_{2}^{-2}x_{3}^{-2}x_{5}^{-1}-3x_{1}^{9}x_{2}^{-2}x_{3}^{-2}x_{5}^{-5}+x_{1}^{11}x_{2}^{-2}x_{3}^{-2}x_{5}^{-7}-x_{1}^{13}x_{2}^{-2}x_{3}^{-2}x_{5}^{-9}$
        \item $h = -x_{1}^{-13}x_{2}^{2}x_{3}^{2}x_{5}^{9}+x_{1}^{-11}x_{2}^{2}x_{3}^{2}x_{5}^{7}-3x_{1}^{-9}x_{2}^{2}x_{3}^{2}x_{5}^{5}-x_{1}^{-5}x_{2}^{2}x_{3}^{2}x_{5}-x_{1}^{-3}x_{2}^{2}x_{3}^{2}x_{5}^{-1}-3x_{1}x_{2}^{2}x_{3}^{2}x_{5}^{-5}+x_{1}^{3}x_{2}^{2}x_{3}^{2}x_{5}^{-7}-x_{1}^{5}x_{2}^{2}x_{3}^{2}x_{5}^{-9}$
        \item $i = -x_{1}^{-5}x_{2}^{-2}x_{3}^{-2}x_{5}^{9}-x_{1}^{-3}x_{2}^{-2}x_{3}^{-2}x_{5}^{7}-x_{1}x_{2}^{-2}x_{3}^{-2}x_{5}^{3}-2x_{1}^{3}x_{2}^{-2}x_{3}^{-2}x_{5}-2x_{1}^{5}x_{2}^{-2}x_{3}^{-2}x_{5}^{-1}-x_{1}^{7}x_{2}^{-2}x_{3}^{-2}x_{5}^{-3}$
        \item $j = -x_{1}^{-7}x_{2}^{2}x_{3}^{2}x_{5}^{3}-2x_{1}^{-5}x_{2}^{2}x_{3}^{2}x_{5}-2x_{1}^{-3}x_{2}^{2}x_{3}^{2}x_{5}^{-1}-x_{1}^{-1}x_{2}^{2}x_{3}^{2}x_{5}^{-3}-x_{1}^{3}x_{2}^{2}x_{3}^{2}x_{5}^{-7}-x_{1}^{5}x_{2}^{2}x_{3}^{2}x_{5}^{-9}$
        \item $k = -x_{1}^{9}x_{2}^{-2}x_{3}^{-2}x_{5}^{-5}-2x_{1}^{11}x_{2}^{-2}x_{3}^{-2}x_{5}^{-7}-2x_{1}^{13}x_{2}^{-2}x_{3}^{-2}x_{5}^{-9}-x_{1}^{15}x_{2}^{-2}x_{3}^{-2}x_{5}^{-11}-x_{1}^{19}x_{2}^{-2}x_{3}^{-2}x_{5}^{-15}-x_{1}^{21}x_{2}^{-2}x_{3}^{-2}x_{5}^{-17}$
        \item $l = -x_{1}^{-21}x_{2}^{2}x_{3}^{2}x_{5}^{17}-x_{1}^{-19}x_{2}^{2}x_{3}^{2}x_{5}^{15}-x_{1}^{-15}x_{2}^{2}x_{3}^{2}x_{5}^{11}-2x_{1}^{-13}x_{2}^{2}x_{3}^{2}x_{5}^{9}-2x_{1}^{-11}x_{2}^{2}x_{3}^{2}x_{5}^{7}-x_{1}^{-9}x_{2}^{2}x_{3}^{2}x_{5}^{5}$
    \end{itemize}
\end{example}

\section*{Acknowledgments} 
The authors wish to extend their heartfelt gratitude to Sam Nelson for his invaluable comments and for his thoughtful responses to our questions regarding key concepts in this study.

\bibliographystyle{plain}
\bibliography{bibliography}

\appendix

\section{The quantum enhancement polynomials for torus links $T(2, q)$.}
\begin{proof}[Proof of Proposition \ref{prop: torus_link}]
 Let $D$ be the standard diagram of $T(2, q)$. 
 Let  $s : \mathcal{C}(D) \to \{A, B\}$ be a state of $D$, and for each state $s$, 
let $k_s$ denote the number of circles in the diagram 
obtained by smoothing each crossing. 
We denote the set of states of $D$ by $\mathcal{S}(D)$. 
    
Suppose first that $q$ is odd. In this case, independent of the choice of orientation of $T(2, q)$, we have
    \[
        k_s = 
        \begin{cases}
            |q| - |s^{-1}(A)| & \text{if } |s^{-1}(A)| \neq |q|, \\
            2 & \text{if } |s^{-1}(A)| = |q|,
        \end{cases}
    \]
and for each $i \, (0 \leq i \leq |q|)$, the following holds: 
\[
    \left|\left\{s \in \mathcal{S}(D) \mid |s^{-1}(A)| = i \right\}\right| = \binom{|q|}{i}.
\]
Each $X_2$-coloring of $D$ can be obtained with two colors $a, b \in X_2$, as shown in Figure~\ref{fig:torus_knot_col}; denote the coloring by $C_{a,b}$. Then, for each $a, b \in X_2$, we have: 
\[
\beta^{(A, B)}_{X_2}(D, C_{a,b}) = w^{-q} \sum_{s \in \mathcal{S}(D)} \delta^{k_s} \prod_{c \in \mathcal{C}(D)} s(c)(a,b,b)^n,
\]
which simplifies to:
\[
    \beta^{(A, B)}_{X_2}(D, C_{a,b}) = w^{-q} \left( A_{a,b,b}^{n|q|}\delta^2 + \sum_{i=1}^{|q|} \binom{|q|}{i} \delta^{|q|-i} A_{a,b,b}^{ni} B_{a,b,b}^{n(|q|-i)} \right),
\]
\[
    = w^{-q}\left(A_{a,b,b}^{n|q|}\delta^2 + (A_{a,b,b}^n + B_{a,b,b}^n\delta)^{|q|} - A_{a,b,b}^{n|q|}\right).
\]
Thus, the assertion holds when $q$ is odd.

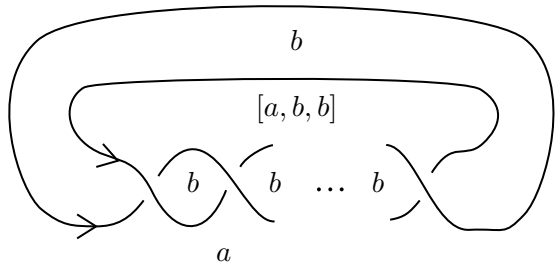
\begin{figure}[htbp]
    \centering
    \includegraphics[width=1.2\textwidth]{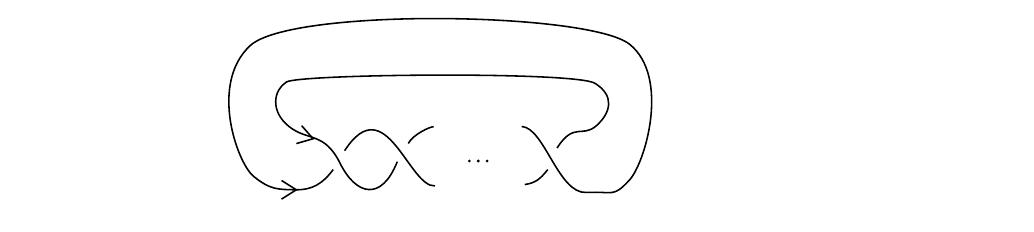}
    \begin{picture}(400,0)(0,0)
        \put(169,46){$b$}
        \put(200,46){$b$}
        \put(239,46){$b$}
        \put(180,20){$a$}
        \put(208,100){$b$}
        \put(195,75){$[a, b, b]$}
    \end{picture}
    \caption{An $X_2$-coloring of the standard diagram of $T(2, q)$, where $q$ is odd.}
    \label{fig:torus_knot_col}
\end{figure}
    
Next, suppose that $q$ is even. Consider first the case where $T(2, q)$ is oriented so that $nq \geq 0$. In this case, we have:
    \[
        k_s = 
        \begin{cases}
            |q| - |s^{-1}(A)| & \text{if } |s^{-1}(A)| \neq |q|, \\
            2 & \text{if } |s^{-1}(A)| = |q|,
        \end{cases}
    \]
and for each $i \, (0 \leq i \leq |q|)$, the following holds:
\[
    \left|\left\{s \in \mathcal{S}(D) \mid |s^{-1}(A)| = i \right\}\right| = \binom{|q|}{i}.
\]
Each $X_2$-coloring of $D$ can be obtained with three colors $a, b, c \in X_2$, as shown in Figure~\ref{fig:torus_link_col_1}; denote the coloring by $C_{a,b,c}$. Then, for each $a, b, c \in X_2$, we have:
\[
\beta^{(A, B)}_{X_2}(D, C_{a,b,c}) = w^{-q} \sum_{s \in \mathcal{S}(D)} \delta^{k_s} \prod_{c \in \mathcal{C}(D)} s(c)(a,b,c)^n,
\]
which simplifies to:
\[
    \beta^{(A, B)}_{X_2}(D, C_{a,b,c}) = w^{-q} \left( A_{a,b,c}^{n\frac{|q|}{2}} A_{a,c,b}^{n\frac{|q|}{2}}\delta^2 + \sum_{i=1}^{\frac{|q|}{2}} \sum_{j=1}^{\frac{|q|}{2}} \binom{\frac{|q|}{2}}{i} \binom{\frac{|q|}{2}}{j} A_{a,b,c}^{ni} A_{a,c,b}^{nj} B_{a,b,c}^{n(\frac{|q|}{2}-i)} B_{a,c,b}^{n(\frac{|q|}{2}-j)}\delta^{|q|-(i+j)} \right),
\]
\[
    = w^{-q} \left( A_{a,b,c}^{n\frac{|q|}{2}} A_{a,c,b}^{n\frac{|q|}{2}}\delta^2 + (A_{a,b,c}^n + B_{a,b,c}^n\delta)^{\frac{|q|}{2}}(A_{a,c,b}^n + B_{a,c,b}^n\delta)^{\frac{|q|}{2}} - A_{a,b,c}^{\frac{|q|}{2}} A_{a,c,b}^{\frac{|q|}{2}} \right).
\]

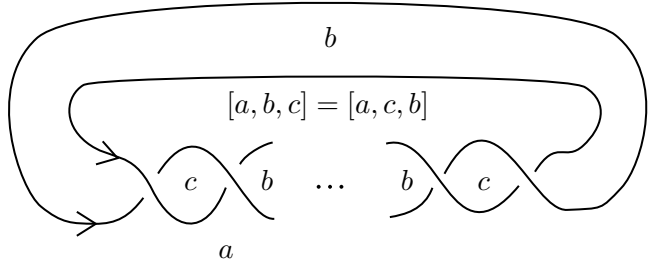
\begin{figure}[htbp]
    \centering
    \includegraphics[width=1.2\textwidth]{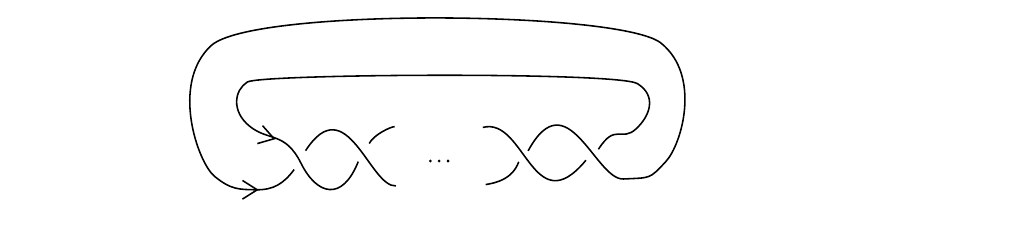}
    \begin{picture}(400,0)(0,0)
        \put(149,46){$c$}
        \put(178,46){$b$}
        \put(231,46){$b$}
        \put(260,46){$c$}
        \put(162,20){$a$}
        \put(202,100){$b$}
        \put(165,75){$[a, b, c] = [a, c, b]$}
    \end{picture}
    \caption{An $X_2$-coloring of the standard diagram of $T(2, q)$, where $q$ is even and $nq \geq 0$.} 
    \label{fig:torus_link_col_1}
\end{figure}

Finally, consider the case where $T(2, q)$ is oriented so that $nq < 0$. In this case, we have:
    \[
        k_s = 
        \begin{cases}
            |q| - |s^{-1}(B)| & \text{if } |s^{-1}(B)| \neq |q|, \\
            2 & \text{if } |s^{-1}(B)| = |q|,
        \end{cases}
    \]
and for each $i \, (0 \leq i \leq |q|)$, the following holds:
\[
    \left|\left\{ s \in \mathcal{S}(D) \mid |s^{-1}(B)| = i \right\}\right| = \binom{|q|}{i}.
\]
Each $X_2$-coloring of $D$ can be obtained with three colors $a, b, c \in X_2$, as shown in Figure~\ref{fig:torus_link_col_2}; denote the coloring by $C_{a,b,c}$. Then, for each $a, b, c \in X_2$, we have:
    \begin{align*}
       \beta^{(A, B)}_{X_2}(D, C_{a,b,c}) &= w^{q} \sum_{s \in \mathcal{S}(D)} \delta^{k_s} \prod_{c \in \mathcal{C}(D)} s(c)(a,b,c)^n \\
                                     &= w^{q} \left( B_{a,b,c}^{n\frac{|q|}{2}} B_{a,c,b}^{n\frac{|q|}{2}} \delta^2 + \sum_{i=1}^{\frac{|q|}{2}} \sum_{j=1}^{\frac{|q|}{2}} \binom{\frac{|q|}{2}}{i} \binom{\frac{|q|}{2}}{j} B_{a,b,c}^{ni} B_{a,c,b}^{nj} A_{a,b,c}^{n(\frac{|q|}{2}-i)} A_{a,c,b}^{n(\frac{|q|}{2}-j)} \delta^{|q|-(i+j)} \right) \\
                                     &= w^{q} \left( B_{a,b,c}^{n\frac{|q|}{2}} B_{a,c,b}^{n\frac{|q|}{2}} \delta^2 + (B_{a,b,c}^n + A_{a,b,c}^n \delta)^{\frac{|q|}{2}} (B_{a,c,b}^n + A_{a,c,b}^n \delta)^{\frac{|q|}{2}} - B_{a,b,c}^{\frac{|q|}{2}} B_{a,c,b}^{\frac{|q|}{2}} \right).
    \end{align*}

    \begin{figure}[htbp]
        \centering
        \includegraphics[width=1.2\textwidth]{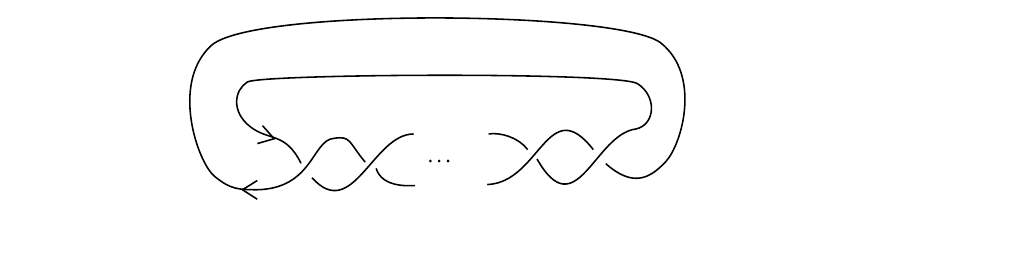}
        \begin{picture}(400,0)(0,0)
            \put(152,56){$c$}
            \put(181,56){$b$}
            \put(234,56){$b$}
            \put(263,56){$c$}

            \put(162,20){$[a, b, c] = [a, c, b]$}

            \put(202,110){$b$}
            \put(202,85){$a$}
        \end{picture}
        \caption{An $X_2$-coloring of the standard diagram of $T(2, q)$, where $q$ is even and $nq < 0$.}
        \label{fig:torus_link_col_2}
    \end{figure}
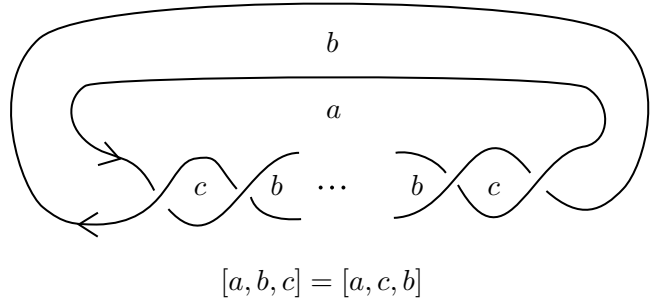
\end{proof}

\section{The quantum enhancement polynomials for twist knots $\mathit{TW}(q)$}
\begin{proof}[Proof of Proposition \ref{prop: twist_knot}]
    Let $D$ be the standard diagram of $\mathit{TW}(q)$. 
    Let $s : \mathcal{C}(D) \to \{A, B\}$ be a state of $D$, and for each state $s$, 
    let $k_s$ denote the number of circles in the diagram 
    obtained by smoothing each crossing. 
    We denote the set of states of $D$ by $\mathcal{S}(D)$. 

    Regardless of the choice of $q$ and orientation, each $X_2$-coloring of $D$ 
    can be obtained with three colors $a, b \in X_2$, as shown in Figure~\ref{fig:twist_knot_col}; we denote the coloring by $C_{a,b}$. 
    \begin{figure}[htpb]
        \centering
        \includegraphics[width=1.2\textwidth]{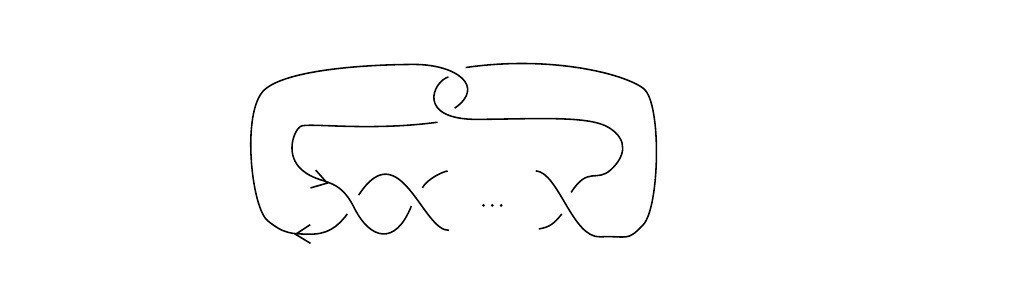}
        \begin{picture}(400,0)(0,0)

            \put(175,56){$a$}
            \put(204,56){$a$}
            \put(248,56){$a$}

            \put(190,20){$b$}
            \put(190,75){$b$}

            \put(208,107){$b$}
            \put(168,107){$a$}
            \put(248,107){$a$}

            \put(205,128){$c_1$}
            \put(205,89){$c_2$}

        \end{picture}
        \caption{An $X_2$-coloring of the standard diagram of $\mathit{TW}(q)$.}
        \label{fig:twist_knot_col}
    \end{figure}
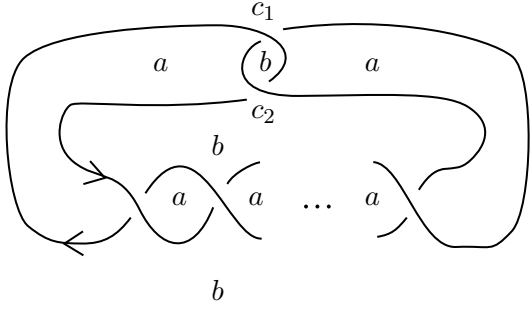
    
    We will discuss the case where $q$ is even, but the same argument applies to the case where $q$ is odd. 
    Let $c_1, c_2$ be the crossings of $D$ shown in Figure~\ref{fig:twist_knot_col}. 
    For each state $s \in \mathcal{S}(D)$, set $S_1 = s(c_1)$ and $S_2 = s(c_2)$. 
    Then, we define $\beta^{(A, B)}_{X_2}(D, C_{a, b}, S_1, S_2)$ by 
    \[
        \beta^{(A, B)}_{X_2}(D, C_{a, b}, S_1, S_2) = w^{n(q+2)} \sum_{s \in \mathcal{S}(D)} \delta^{k_s} \prod_{c \in \mathcal{C}(D) \setminus \{c_1, c_2\}} s(c)(a,b,b)^n S_1(a,b,b)^n S_2(a,b,b)^n.
    \]
    Note that 
    \[
    \beta^{(A, B)}_{X_2}(D, C_{a, b}) = \beta^{(A, B)}_{X_2}(D, C_{a, b}, A, A) + \beta^{(A, B)}_{X_2}(D, C_{a, b}, A, B) + \beta^{(A, B)}_{X_2}(D, C_{a, b}, B, A) + \beta^{(A, B)}_{X_2}(D, C_{a, b}, B, B).
    \]
    
    For each state $s \in \mathcal{S}(D)$, set $m^s_A = \left|\left\{c \in \mathcal{C}(D) \setminus \{c_1, c_2\} \mid s(c) = A \right\}\right|$. 
    
    Consider first the case where $S_1 = A$ and $S_2 = A$.  
    In this case, we have
    \[
        k_s = 
        \begin{cases}
            m^s_A + 1 & \text{if } m^s_A \neq 0, \\
            3 & \text{if } |s^{-1}(A)| =  0, 
        \end{cases}
    \] 
    and therefore,
    \begin{align*}
         \beta^{(A, B)}_{X_2}(D, C_{a, b}, A, A) &= w^{n(q+2)} \sum_{s \in \mathcal{S}(D)} \delta^{k_{s}} \prod_{c \in \mathcal{C}(D) \setminus \{c_1, c_2\}} s(c)(a,b,b)^n A(a,b,b)^n A(a,b,b)^n \\
         &= w^{n(q+2)} \left(\sum_{i=0}^{|q|} \binom{|q|}{i} \delta^{i+1}A_{a,b,b}^{ni}B_{a,b,b}^{n(|q|-i)} - \delta B_{a,b,b}^{n|q|} + \delta^3 B_{a,b,b}^{n|q|}\right)A_{a,b,b}^{2n} \\
         &= w^{n(q+2)} \left((A_{a,b,b}^n\delta + B_{a,b,b}^n)^{|q|} + B_{a,b,b}^{n|q|}(\delta^2 - 1)\right)\delta A_{a,b,b}^{2n}.
    \end{align*}

    Next, consider the case where $S_1 = A$ and $S_2 = B$, or $S_1 = A$ and $S_2 = A$.  
    In this case, we have 
    \[
        k_s = 
        \begin{cases}
            m^s_A & \text{if } m^s_A \neq 0, \\
            2 & \text{if } |s^{-1}(A)| =  0, 
        \end{cases}
    \]
    and therefore, 
    \begin{align*}
         \beta^{(A, B)}_{X_2}(D, C_{a, b}, S_1, S_2) &= w^{n(q+2)} \sum_{s \in \mathcal{S}(D)} \delta^{k_s} \prod_{c \in \mathcal{C}(D) \setminus \{c_1, c_2\}} s(c)(a,b,b)^n S_1(a,b,b)^n S_2(a,b,b)^n \\
         &= w^{n(q+2)} \left(\sum_{i=0}^{|q|} \binom{|q|}{i} \delta^i A_{a,b,b}^{ni} B_{a,b,b}^{n(|q|-i)} - B_{a,b,b}^{n|q|} + \delta^2 B_{a,b,b}^{n|q|}\right) A_{a,b,b}^n B_{a,b,b}^n \\
         &= w^{n(q+2)} \left((A_{a,b,b}^n\delta + B_{a,b,b}^n)^{|q|} + B_{a,b,b}^{n|q|}(\delta^2 - 1)\right) A_{a,b,b}^n B_{a,b,b}^n.
    \end{align*}

    Finally, consider the case where $S_1 = B$ and $S_2 = B$. 
    In this case, $k_s = m^s_A + 1$. Thus, we have 
    \begin{align*}
         \beta^{(A, B)}_{X_2}(D, C_{a, b}, B, B) &= w^{n(q+2)} \sum_{s \in \mathcal{S}(D)} \delta^{k_{s}} \prod_{c \in \mathcal{C}(D) \setminus \{c_1, c_2\}} s(c)(a,b,b)^n A(a,b,b)^n A(a,b,b)^n \\
         &= w^{n(q+2)} \sum_{i=0}^{|q|} \binom{|q|}{i} \delta^{i+1} A_{a,b,b}^{ni} B_{a,b,b}^{n(|q|-i)} B_{a,b,b}^{2n} \\
         &= w^{n(q+2)} \left(((A_{a,b,b}^n\delta + B_{a,b,b}^n)^{|q|} + B_{a,b,b}^{n|q|}(\delta^2 - 1))\delta B_{a,b,b}^{2n} - \delta B_{a,b,b}^{n(|q|+2)}\right).
    \end{align*}

    From the above calculations, we obtain 
    \begin{align*}
        \beta^{(A, B)}_{X_2}(D, C_{a, b}) = w^{n(q+2)} & \left(((A_{a,b,b}^n\delta + B_{a,b,b}^n)^{|q|} + B_{a,b,b}^{n|q|} (\delta^2 - 1)) \right. \\
        &\left. \cdot (\delta A_{a,b,b}^{2n} + 2A_{a,b,b}^n B_{a,b,b}^n + \delta B_{a,b,b}^{2n}) - \delta B_{a,b,b}^{n(|q|+2)}\right).
    \end{align*}
\end{proof}

\section{The quantum enhancement polynomials for Thistlethwaite's link}

Let $(A^{(i)}, B^{(i)})$ ($i \in \{1, \ldots, 5\}$) be the universal tribracket brackets with respect to $X_2$. 
Let $L$ be Thistlethwaite's link, as shown in Figure~\ref{fig:thistlethwaite}. 
Then, for each $i \in \{1,2,3,4,5\}$, $\Phi^{(A^{(i)}, B^{(i)})}_{X_2}(L)$ is computed as follows: 
\begin{itemize}
    \item $\Phi^{(A^{(1)}, B^{(1)})}_{X_2}(L) = \Phi^{(A^{(2)}, B^{(2)})}_{X_2}(L) = 4u^a + 4u^b + 4u^c + 4u^d$.
    \item $\Phi^{(A^{(3)}, B^{(3)})}_{X_2}(L) = \Phi^{(A^{(4)}, B^{(4)})}_{X_2}(L) = 4u^a + 2u^c + 2u^d + 2u^e + 2u^f + 2u^g + 2u^h$.
    \item $\Phi^{(A^{(5)}, B^{(5)})}_{X_2}(L) = 8u^a + 4u^i + 4u^j$,
\end{itemize}
where $a,b,c,d,e,f,g,h,i,j \in \mathbb{Z}[x_1,x_2,x_3,x_5]$ are defined as follows: 
\begin{itemize}
    \item $a = -x_1^{-3}x_5^3 - 3x_1^{-1}x_5 - 3x_1x_5^{-1} - x_1^3x_5^{-3}$
    \item $b = x_1^{-7}x_5^7 - x_1^{-3}x_5^3 - 3x_1^{-1}x_5 - 4x_1x_5^{-1} - x_1^3x_5^{-3} - x_1^5x_5^{-5} + x_1^{13}x_5^{-13}$
    \item $c = -x_1^{-5}x_2^2x_3^2x_5 - 2x_1^{-3}x_2^2x_3^2x_5^{-1} + 2x_1^{-1}x_2^2x_3^2x_5^{-3} - 2x_1x_2^2x_3^2x_5^{-5} - 2x_1^3x_2^2x_3^2x_5^{-7} - x_1^5x_2^2x_3^2x_5^{-9} + x_1^7x_2^2x_3^2x_5^{-11} - x_1^9x_2^2x_3^2x_5^{-13}$
    \item $d = -x_1^{-1}x_2^{-2}x_3^{-2}x_5^5 - 2x_1x_2^{-2}x_3^{-2}x_5^3 - 2x_1^3x_2^{-2}x_3^{-2}x_5 - 2x_1^5x_2^{-2}x_3^{-2}x_5^{-1} - 2x_1^7x_2^{-2}x_3^{-2}x_5^{-3} - x_1^9x_2^{-2}x_3^{-2}x_5^{-5} + x_1^{11}x_2^{-2}x_3^{-2}x_5^{-7} + x_1^{13}x_2^{-2}x_3^{-2}x_5^{-9}$
    \item $e = -x_1^{-17}x_5^{17} + x_1^{-15}x_5^{15} + 2x_1^{-13}x_5^{13} + x_1^{11}x_5^{-11} - 3x_1^7x_5^{-7} - 2x_1^{-5}x_5^5 - 2x_1^{-3}x_5^3 - 2x_1^{-1}x_5 - x_1x_5^{-1} - x_1^3x_5^{-3}$
    \item $f = -x_1^{-3}x_5^3 - x_1^{-1}x_5 - 2x_1x_5^{-1} - 2x_1^3x_5^{-3} - 2x_1^5x_5^{-5} - 3x_1^7x_5^{-7} + x_1^{11}x_5^{-11} + 2x_1^{13}x_5^{-13} + x_1^{15}x_5^{-15} - x_1^{17}x_5^{-17}$
    \item $g = x_1^{-25}x_2^2x_3^2x_5^{21} + x_1^{-23}x_2^2x_3^2x_5^{19} - x_1^{-21}x_2^2x_3^2x_5^{17} - 2x_1^{-19}x_2^2x_3^2x_5^{15} - 2x_1^{-17}x_2^2x_3^2x_5^{13} - 2x_1^{-15}x_2^2x_3^2x_5^{11} - 2x_1^{-13}x_2^2x_3^2x_5^9 - x_1^{-11}x_2^2x_3^2x_5^7$
    \item $h = x_1^3x_2^{-2}x_3^{-2}x_5 + x_1^5x_2^{-2}x_3^{-2}x_5^{-1} - x_1^7x_2^{-2}x_3^{-2}x_5^{-3} - 2x_1^9x_2^{-2}x_3^{-2}x_5^{-5} - 2x_1^{11}x_2^{-2}x_3^{-2}x_5^{-7} - 2x_1^{13}x_2^{-2}x_3^{-2}x_5^{-9} - 2x_1^{15}x_2^{-2}x_3^{-2}x_5^{-11} - x_1^{17}x_2^{-2}x_3^{-2}x_5^{-13}$
    \item $i = -x_1^{-7}x_2^2x_3^2x_5^3 - 3x_1^{-5}x_2^2x_3^2x_5 - 3x_1^{-3}x_2^2x_3^2x_5^{-1} - x_1^{-1}x_2^2x_3^2x_5^{-3}$
    \item $j = -x_1x_2^{-2}x_3^{-2}x_5^3 - 3x_1^3x_2^{-2}x_3^{-2}x_5 - 3x_1^5x_2^{-2}x_3^{-2}x_5^{-1} - x_1^7x_2^{-2}x_3^{-2}x_5^{-3}$
\end{itemize}

\end{document}